\documentclass[fleqn,final]{zzz}
\usepackage{mathrsfs}
\usepackage{showkeys} 
\usepackage{color}
\usepackage{tikz}
\usepackage{amsmath,amsthm,amsbsy,amssymb}
\usepackage{graphicx}
\usepackage{rotating}
\usepackage{fixmath}
\usepackage[pdftex]{hyperref}
\usepackage{showkeys}
\usepackage{soul}
\usepackage{bm}
\usepackage{mathrsfs}

\hypersetup{
 colorlinks = true,
 urlcolor = blue,
 linkcolor = red,
 citecolor = green,
}

\sethlcolor{black}

\newcommand\ppm{\{\!\pm\!\}}

\newcommand{\Whyp}[5]{\,\mbox{}_{#1}W_{#2}\!\left({#3};{#4};{#5}\right)}
\newcommand{\qhyp}[5]{\,\mbox{}_{#1}\phi_{#2}\!\left(
\genfrac{}{}{0pt}{}{#3}{#4};#5\right)}

\newcommand{\hyp}[5]{\,\mbox{}_{#1}F_{#2}\!\left(
 \genfrac{}{}{0pt}{}{#3}{#4};#5\right)}

\newtheorem{thm}{Theorem}[section]
\newtheorem{cor}[thm]{Corollary}

\newtheorem{rem}[thm]{Remark}
\newtheorem{lem}[thm]{Lemma}

\makeatletter
\def\eqnarray{\stepcounter{equation}\let\@currentlabel=\theequation
\global\@eqnswtrue
\tabskip\@centering\let\\=\@eqncr
$$\halign to \displaywidth\bgroup\hfil\global\@eqcnt\z@
$\displaystyle\tabskip\z@{##}$&\global\@eqcnt\@ne
\hfil$\displaystyle{{}##{}}$\hfil
&\global\@eqcnt\tw@ $\displaystyle{##}$\hfil
\tabskip\@centering&\llap{##}\tabskip\z@\cr}

\def\endeqnarray{\@@eqncr\egroup
\global\advance\c@equation\m@ne$$\global\@ignoretrue}

\def\@yeqncr{\@ifnextchar [{\@xeqncr}{\@xeqncr[5pt]}}
\makeatother

\newcommand{\Z}{\mathbb{Z}} 
\newcommand{\R}{\mathbb{R}} 
\newcommand{\C}{\mathbb{C}} 
\newcommand{\N}{\mathbb{N}} 

\newcommand{\CC}{{{\mathbb C}}}
\newcommand{\CCast}{{{\mathbb C}^\ast}}

\newcommand{\expe}{{\mathrm e}}

\newcommand{\dd}{{\mathrm d}}

\newcommand{\midtilde}{\raisebox{-0.25\baselineskip}{\textasciitilde}}

\usepackage{scalerel,url}
\let\svus_
\catcode`_=\active %
\def_{\ifmmode\expandafter\svus\expandafter\bgroup\expandafter\lowerit
\else\svus\fi}
\def\lowerit#1{\ThisStyle{\raisebox{-2\LMpt}{$\SavedStyle#1$}}\egroup}
\makeatletter
 \AtBeginDocument{
 \check@mathfonts
 \fontdimen13\textfont2=3.5pt
 \fontdimen14\textfont2=3.5pt
 \fontdimen16\textfont2=2.5pt
 \fontdimen17\textfont2=2.5pt
 }
\makeatother

\allowdisplaybreaks[1]

\begin{document}

\renewcommand{\PaperNumber}{***}

\FirstPageHeading

\ShortArticleName{Product formulas for basic hypergeometric series}

\ArticleName{Product formulas for basic hypergeometric series\\by
evaluations of Askey--Wilson polynomials}

\Author{Howard S. Cohl$\,^{\ast}\orcidB{}$ and Michael J. Schlosser$\,^{\dag}\orcidC{}$ 
}

\AuthorNameForHeading{H.~S.~Cohl}
\Address{$^\ast$ Applied and Computational 
Mathematics Division, National Institute of Standards 
and Tech\-no\-lo\-gy, Gaithersburg, MD 20899-8910, USA
\URLaddressD{
\href{http://www.nist.gov/itl/math/msg/howard-s-cohl.cfm}
{http://www.nist.gov/itl/math/msg/howard-s-cohl.cfm}
}
} 
\EmailD{howard.cohl@nist.gov} 

\AuthorNameForHeading{H.~S.~Cohl, M.~J.~Schlosser}
\Address{$^\dag$ Fakult\"at f\"ur Mathematik,
Universit\"at Wien,
Oskar-Morgenstern-Platz 1, Vienna,
Austria
\URLaddressD{
\href{https://www.mat.univie.ac.at/~schlosse/}
{https://www.mat.univie.ac.at/\midtilde{}schlosse/}
}
} 
\EmailD{michael.schlosser@univie.ac.at} 


\ArticleDates{Received \today~in final form ????; Published online ????}

\Abstract{Ismail and Wilson derived a generating function for Askey--Wilson polynomials which is given by a product of $q$-Gauss (Heine) nonterminating basic hypergeometric functions. We provide a generalization of that generating function which contains an extra parameter. A special case gives a closed form summation formula for a quadruple basic hypergeometric sum.
We further present new terminating balanced ${}_4\phi_3$ summations that give rise to $q$-quadratic special values for Askey--Wilson polynomials. We also similarly present new terminating 2-balanced and 3-balanced ${}_4\phi_3$ summations.
Using the Ismail--Wilson generating function combined with explicit summations for terminating balanced basic hypergeometric $_4\phi_3$ series,
we compute new basic hypergeometric product transformations for nonterminating basic hypergeometric series and provide corresponding integral representations. Further new identities are obtained by applying
Cayley--Orr type expansion formulas.}

\Keywords{Nonterminating basic hypergeometric series; transformations; products of basic hypergeometric series}

\Classification{33D15; 33D50}

\begin{flushright}
\begin{minipage}{60mm}
\it Dedicated to George E.~Andrews 
and Bruce C.~Berndt on the occasion of their 85\textsuperscript{th} birthdays.
\end{minipage}
\end{flushright}

\section{Introduction}
The Askey--Wilson polynomials stand on the top of the
  $q$-Askey-scheme \cite{Koekoeketal}
as they are the most general family of
orthogonal polynomials in one variable that can be written as
basic hypergeometric series. They are well-studied objects
and play a central role in $q$-series and the theory
of ($q$-)special functions. The monograph \cite{Ismail} discusses
many of the relevant properties of the Askey--Wilson polynomials
and related objects, together with detailed proofs.
In \cite[(1.9)]{IsmailWilson82}, Ismail and Wilson obtained
a generating function for the Askey--Wilson polynomials,
see \eqref{AWgf} below;
this can be used to uniquely determine the coefficients
in the three-term relation satisfied by the Askey--Wilson polynomials
(see \cite[pp.~385--386]{Ismail}).

This paper is devoted to a thorough study of the generating function \eqref{AWgf}, viewed as an expansion of a product of two
nonterminating basic hypergeometric $_2\phi_1$ series in terms of a
single sum of the Askey--Wilson polynomials with coefficients
being basic hypergeometric terms. One of our main results is
a generalization of the Ismail--Wilson generating function in the
form of a product formula for two nonterminating basic
hypergeometric $_3\phi_2$ series.
The expansion here takes the form of a three-fold sum, see
Theorem~\ref{thm:32pf}. An immediate consequence of this formula
is an elegant closed
form evaluation for a specific quadruple basic hypergeometric sum,
given in Corollary~\ref{cor:1.3}, which we believe is new.
In addition, a great portion of our study is devoted to
finding applications of the
generating function \eqref{AWgf} in the context of product formulas for
basic hypergeometric series (which we deduce by applying various
available summations for terminating balanced $_4\phi_3$ series)
and corresponding integral representations.

We believe that our study, which is focused on
formulas connecting (products of) $_2\phi_1$ series and balanced terminating
$_4\phi_3$ series, 
is a valuable contribution to
the theory of basic hypergeometric series. 
These functions are two of the most important families of 
$q$-special functions. They satisfy important properties including difference equations that extend the hypergeometric differential 
equation, and in particular, their appearance in the
theory of basic hypergeometric orthogonal polynomials, cf.~\cite{Ismail}.
We hope that the identities
given here will be useful, for instance in studies related
to combinatorial representation theory, where product
and linearization formulas typically appear.

Before turning to the Askey--Wilson polynomials in Section~\ref{sec:AW},
we recall some basic notions from the theory of basic hypergeometric series
in the following Subsection~\ref{subsec:prel}.
Further material concerns an overview of $_4\phi_3$ summations in
Section~\ref{sec:4phi3}, and, as applications, the
nonterminating product transformations in Section~\ref{sec:npt} and 
a short conclusion is given in the end in Section~\ref{sec:con}.

\subsection{Preliminaries}
\label{subsec:prel}
We adopt the following set 
notations:~$\N_0:=\{0\}\cup\N=\{0, 1, 2, ...\}$, and we use the sets $\Z$, $\R$, $\CC$ which represent 
the integers, real numbers and complex numbers respectively,
{and further write}
$\CCast:=\CC\setminus\{0\}$.
Throughout the paper, we assume that the empty sum 
vanishes and the 
empty product 
is unity.
We also adopt the following conventions for succinctly 
writing elements of lists. To indicate sequential positive \textit{and} negative 
elements and powers, we write
\[
\ppm a:=\{a,-a\} \qquad\text{and}\qquad z^{\ppm}:=\{ z,z^{-1}\}.
\]
In order to obtain our derived identities, we rely on properties 
of the $q$-Pochhammer symbol (a.k.a.\ the $q$-shifted factorial). 
For any $n\in \N_0$, $q,a \in \CC$, 
the 
$q$-Pochhammer symbol is defined as
\[
(a;q)_n:=(1-a)(1-aq)\cdots(1-aq^{n-1}),
 \qquad\text{and }\qquad (a;q)_\infty:=\lim_{n\to\infty}(a;q)_n,
 \]
 {where in the infinite product we require $|q|<1$ for
absolute convergence}. 
We will also use, for $k\in\N$ and $n\in\N_0\cup\{\infty\}$,
the common notational product convention
\[
(a_1,...,a_k;q)_n:=(a_1;q)_n\cdots(a_k;q)_n.
\]
The $q$-Pochhammer symbol satisfies a number of elementary
identities (cf.\ \cite[Appendix~I]{GaspRah}) such as
\begin{equation}
(a^2;q^2)_n=(\ppm a;q)_n,
\label{sqPoch}
\end{equation}
which we tacitly use throughout this paper.
Some other useful identities for the $q$-shifted factorial for $n\in\N_0$, are
\begin{eqnarray*}
&&\hspace{-6.5cm}(a;q)_{n+1}=(1-a)(qa;q)_n,\\
&&\hspace{-6.5cm}(a;q)_{n-1}=\frac{(q^{-1}a;q)_n}{(1-q^{-1}a)},\\
&&\hspace{-6.5cm}\frac{(1-q^na)}{(1-a)}=\frac{(qa;q)_n}{(a;q)_n},\\
&&\hspace{-6.5cm}(a;q^{\frac12})_n=\left\{ \begin{array}{ll}
\displaystyle (a,q^\frac12a;q)_{\frac{n}{2}} & \qquad\mathrm{if}\ n\ \mathrm{even},\\[9pt]
\displaystyle 
(1-a)(q^{\frac12}a,qa;q)_{\frac{n-1}{2}} 
& \qquad\mathrm{if}\ n\ \mathrm{odd},
\end{array} \right.\\[9pt]
&&\hspace{-6.5cm}(a;q)_{2n}=(a,qa;q^2)_n=(\ppm \sqrt{a},\ppm\sqrt{qa};q)_n.
\end{eqnarray*}
The basic hypergeometric series, which we 
will often use, is defined for
$q,z\in\CCast$, such that $|q|<1$, $s,r\in\N_0$, 
$b_j\not\in q^{-\mathbb N_0}$, 
$j=1,...,s$, as
\cite[(1.10.1)]{Koekoeketal}
\begin{equation*}
\qhyp{r}{s}{a_1,...,a_r}
{b_1,...,b_s}
{q,z}
:=\sum_{k=0}^\infty
\frac{(a_1,...,a_r;q)_k}
{(q,b_1,...,b_s;q)_k}
\left((-1)^kq^{\binom k2}\right)^{1+s-r}
z^k,
\end{equation*}
where for $r=s+1$ we assume $|z|<1$ for convergence, if the series
does not terminate. (For $r\le s$ the series converges for any $z$,
due to the quadratic powers of $q$ appearing in the terms of the series.)
We refer to a basic hypergeometric
series as {\it $\ell$-balanced} if
$q^\ell a_1\cdots a_r=b_1\cdots b_s$, 
and {\it balanced} (or Saalsch\"utzian) if $\ell=1$.

The basic hypergeometric function $_r\phi_s$ is a $q$-extension of the
generalized hypergeometric function ${}_rF_s$, defined for
$z,a_1,\ldots,a_r\in\C$ and $b_1,\ldots,b_s\in\C\setminus\{0,-1,-2,\ldots\}$ as
\cite[\href{http://dlmf.nist.gov/16.2.E1}{(16.2.1)}]{NIST:DLMF}
\begin{equation*}
\hyp{r}{s}{a_1,\ldots,a_r}{b_1,\ldots,b_s}{z}:=
\sum_{k=0}^\infty\frac{(a_1)_k\cdots(a_r)_k}{(b_1)_k\cdots(b_s)_k}
\frac{z^k}{k!},
\end{equation*}
where
$(a)_k=a(a+1)\cdots(a+k-1)$
is the shifted factorial. For the conditions of absolute or conditional
convergence of the ${}_rF_s$ series, in case it does not terminate,
see \cite[\href{http://dlmf.nist.gov/16.2}{Section 16.2}]{NIST:DLMF}.

\section{The Askey--Wilson polynomials}\label{sec:AW}
The Askey--Wilson polynomials~\cite[Chapter~15]{Ismail} form the most general
family of orthogonal polynomials in $x$ that are also basic hypergeometric
series. They contain, in addition to the variable $x$ and the base $q$,
four extra parameters $a,b,c,d$. The $n$-th degree Askey--Wilson
polynomial is denoted by $p_n(x;a,b,c,d|q)$.
These polynomials are symmetric in the four
variables $a,b,c,d$ and have the following
terminating balanced ${}_4\phi_3$ basic hypergeometric series
representations
\cite[(15)]{CohlCostasSantos20b}
(with obvious additional ones corresponding to the symmetries
in the four variables), see \cite[Theorem~7]{CohlCostasSantos20b}.
\begin{thm}
\label{AWthm}
Let $n\in\N_0$, $x=\frac12(w+w^{-1})$,
and $q,w,a,b,c,d\in\CCast$. 
Then
\begin{eqnarray}
\label{aw:def1} &&\hspace{-1.2cm}p_n(x;a,b,c,d|q) {}= a^{-n} (ab,ac,ad;q)_n 
\qhyp43{q^{-n},q^{n-1}abcd, aw^{\ppm}}{ab,ac,ad}{q,q}\\
\label{aw:def2}
 &&\hspace{1.45cm}{}=q^{-\binom{n}{2}} (-a)^{-n} 
\frac{(\frac{abcd}{q};q)_{2n}
(a w^{\ppm};q)_n}
{(\frac{abcd}{q};q)_n}
\qhyp43{q^{-n},
\frac{q^{1-n}}{ab},
\frac{q^{1-n}}{ac},
\frac{q^{1-n}}{ad}
}
{\frac{q^{2-2n}}{abcd},\frac{q^{1-n}}{a}w^{\ppm}}{q,q}\\
\label{aw:def3} && \hspace{1.45cm}{}=w^n\Big(ab,\frac cw,\frac dw;q\Big)_n
\qhyp43{q^{-n},aw,bw,\frac{q^{1-n}}{cd}}
{ab,q^{1-n}\frac{w}c,q^{1-n}\frac{w}d}{q,q}.
\end{eqnarray}
\end{thm}

An important product generating function for Askey--Wilson polynomials was derived by Ismail--Wilson, namely \cite[(1.9)]{IsmailWilson82}
\begin{eqnarray}
&&\hspace{-3.5cm}\sum_{n=0}^\infty
\frac{t^n\,p_n(x;a,b,c,d|q)}{(q,ab,cd;q)_n}=\qhyp21{aw,bw}{ab}{q,\frac{t}{w}}\qhyp21{\frac{c}{w},\frac{d}{w}}{cd}{q,tw},
\label{AWgf}
\end{eqnarray}
where $x=\frac12(w+w^{-1})$.
Comparison of the coefficients of $t^n/(q,ab,cd;q)_n$
on both sides of \eqref{AWgf} reveals that
the Askey--Wilson polynomials satisfy the convolution formula
\begin{equation}
p_n(x;a,b,c,d|q)=(q,ab,cd;q)_n\sum_{j=0}^n\frac{(aw,bw;q)_j}{(q,ab;q)_j}
\frac{(\frac cw,\frac dw;q)_{n-j}}{(q,cd;q)_{n-j}}w^{n-2j},
\label{AWcf}
\end{equation}
which, after some elementary manipulations of $q$-Pochhammer symbols
is just \eqref{aw:def3}. (This shows that \eqref{AWcf} and \eqref{aw:def3} are equivalent.) 

\begin{rem}
A referee pointed out that \eqref{AWcf} with $w=\expe^{i\theta}$ yields for $a=b=c=d=0$ the well-known formula for continuous $q$-Hermite polynomials given at the end of \cite[Section 14.26]{Koekoeketal}. This is interesting because, for instance in the standard $q$-hypergeometric expression for the Askey--Wilson polynomials \eqref{aw:def1}, one cannot specialize $a=b=c=d=0$.
\end{rem}

\subsection{A triple sum generating function for Askey--Wilson polynomials}

We have the following triple summation generating function for
Askey--Wilson polynomials
which generalizes the Ismail--Wilson product generating function \eqref{AWgf}
{by an extra parameter $u$}.
\begin{thm}
\label{thm:32pf}
Let {$|q|<1$ and} $u,w,t,a,b,c,d\in\CCast$ such that
{$|u|<\min(|a|,|c|)$,} $|t|<\min(|w|,|w^{-1}|)$, $x=\frac12(w+w^{-1})$. Then
\begin{eqnarray}\label{32pf}
&&\hspace{-2.5cm}\qhyp32{\frac{u}{t},aw,bw}{ab,uw}{q,\frac{t}{w}}
\qhyp32{\frac{u}{t},\frac cw,\frac dw}{cd,\frac uw}{q,tw}\nonumber\\
&&\hspace{-0.3cm}=\frac{(\frac{u}{a},\frac{u}{c};q)_\infty}{(uw^{\ppm};q)_\infty}
\sum_{n,k,l\ge 0}\frac{t^nu^{k+l}}{a^kc^l}
\frac{(aw^{\ppm};q)_k(cw^{\ppm};q)_l\,p_n(x;q^ka,b,q^lc,d|q)}{(q;q)_k(q;q)_l(q;q)_n(ab;q)_{n+k}(cd;q)_{n+l}}.
\end{eqnarray}
\end{thm}
\begin{proof}
The main idea is to establish the identity by comparison of the
coefficients of $t^n$ on both sides of the identity.
Since
\[
\Big(\frac ut;q\Big)_kt^k=\sum_{j=0}^k(-1)^{k-j}q^{\binom{k-j}2}
\frac{(q;q)_k}{(q;q)_j(q;q)_{k-j}}u^{k-j}t^j,
\]
by the terminating $q$-binomial theorem \cite[Ex.~1.2(vi)]{GaspRah}, we see
that the first $_3\phi_2$ series on the left-hand side of \eqref{32pf} is
\begin{align*}
\qhyp32{\frac{u}{t},aw,bw}{ab,uw}{q,\frac{t}{w}}&=
\sum_{k=0}^\infty\frac{(aw,bw;q)_k}{(q,ab,uw;q)_k}w^{-k}
\sum_{j=0}^k(-1)^{k-j}q^{\binom{k-j}2}
\frac{(q;q)_k}{(q;q)_j(q;q)_{k-j}}u^{k-j}t^j\\
 &=\sum_{j=0}^\infty\frac{(aw,bw;q)_j}{(q,ab,uw;q)_j}w^{-j}t^j
\sum_{k=0}^\infty
 \frac{(q^jaw,q^jbw;q)_k}{(q,q^jab,q^juw;q)_k}
 (-1)^kq^{\binom k2}w^{-k}u^k. 
\end{align*}
Now, the inner sum over $k$ is a $_2\phi_2$ series which can be transformed
to a multiple of a $_2\phi_1$ series by Jackson's
transformation in \cite[(III.4)]{GaspRah}, which
we restate for convenience in the following form:
\begin{equation}\label{eq:J22}
\qhyp22{a,\frac cb}{c,az}{q,bz}=
\frac{(z;q)_\infty}{(az;q)_\infty}\qhyp21{a,b}{c}{q,z},
\end{equation}
valid for $|z|<1$ for convergence of the series on the right-hand side.
We thus apply the $(a,b,c,z)\mapsto{(awq^j,a/w,abq^j,u/a)}$ case
of \eqref{eq:J22} and obtain for the first $_3\phi_2$ series
on the left-hand side of \eqref{32pf} 
\begin{equation*}
\qhyp32{\frac{u}{t},aw,bw}{ab,uw}{q,\frac{t}{w}}
=\sum_{j=0}^\infty\frac{(aw,bw;q)_j}{(q,ab,uw;q)_j}w^{-j}t^j
\frac{(\frac ua;q)_\infty}{(q^juw;q)_\infty}
\sum_{k=0}^\infty
\frac{(q^jaw,\frac aw;q)_k}{(q,q^jab;q)_k}a^{-k}u^k.
\end{equation*}
Thus, the coefficient of $t^j$ of the first $_3\phi_2$ series
on the left-hand side of \eqref{32pf} is
\[
\frac{(\frac ua;q)_\infty}{(uw;q)_\infty}\frac{(aw,bw;q)_j}{(q,ab;q)_j}w^{-j}
\sum_{k=0}^\infty\frac{(q^jaw,\frac aw;q)_k}
{(q,q^jab;q)_k}a^{-k}u^k.
\]
Similarly, by symmetry, the coefficient of $t^{n-j}$ of the second
$_3\phi_2$ series on the left-hand side of \eqref{32pf} is
\[
\frac{(\frac uc;q)_\infty}{(\frac uw;q)_\infty}
\frac{(\frac cw,\frac dw;q)_{n-j}}{(q,cd;q)_{n-j}}w^{n-j}
\sum_{l=0}^\infty\frac{(q^{n-j}\frac{c}w,cw;q)_l}
{(q,q^{n-j}cd;q)_l}c^{-l}u^l.
\]
In total, the coefficient of $t^n$ of the product of the two
$_3\phi_2$ series on the left-hand side of \eqref{32pf} is
\begin{eqnarray*}
&&\hspace{-0.5cm}\frac{(\frac ua,\frac uc;q)_\infty}{(uw,\frac uw;q)_\infty}\sum_{j=0}^n
\frac{(aw,bw;q)_j}{(q,ab;q)_j}
\frac{(\frac cw,\frac dw;q)_{n-j}}{(q,cd;q)_{n-j}}w^{n-2j}
\sum_{k,l=0}^\infty\frac{(q^jaw,\frac aw;q)_k}
{(q,q^jab;q)_k}\frac{(q^{n-j}\frac {c}w,cw;q)_l}
{(q,q^{n-j}cd;q)_l}\frac{u^{k+l}}{a^kc^l}\\
&&=\frac{(\frac ua,\frac uc;q)_\infty}{(uw,\frac uw;q)_\infty}
\sum_{k,l=0}^\infty\frac{(aw,\frac aw;q)_k}
{(q,ab;q)_k}\frac{(cw,\frac cw;q)_l}
{(q,cd;q)_l}\frac{u^{k+l}}{a^kc^l}
\sum_{j=0}^n\frac{(q^kaw,bw;q)_j}{(q,q^kab;q)_j}
\frac{(q^l\frac{c}w,\frac dw;q)_{n-j}}{(q,q^lcd;q)_{n-j}}w^{n-2j}.
\end{eqnarray*} 
Now the last sum over $j$ is a multiple of an Askey--Wilson polynomial
by virtue of \eqref{AWcf}. We thus obtain the following expression
for the coefficient of $t^n$ of the left-hand side of \eqref{32pf}:
\begin{equation*}
\frac{(\frac ua,\frac uc;q)_\infty}{(uw,\frac uw;q)_\infty}
\sum_{k,l=0}^\infty\frac{(aw,\frac aw;q)_k}
{(q,ab;q)_k}\frac{(cw,\frac cw;q)_l}
{(q,cd;q)_l}\frac{u^{k+l}}{a^kc^l}
\frac{p_n(x;q^ka,b,q^lc,d|q)}{(q,q^kab,q^lcd;q)_n},
\end{equation*}
and the result follows.
\end{proof}

\noindent An immediate relevant consequence of Theorem \ref{thm:32pf} is
obtained by letting $u=t$, in which case both ${}_3\phi_2$ basic hypergeometric series reduce to unity. This gives the following closed form summation formula for
a triple sum; or, after writing the Askey--Wilson polynomial as a sum,
say, using \eqref{AWcf}, for a basic hypergeometric \textit{quadruple} sum:
\begin{cor}
\label{cor:1.3}
Let $|q|<1$ and $w,t,a,b\in\CCast$ such that
$|t|<\min(|a|,|c|,|w|,|w^{-1}|)$, $x=\frac12(w+w^{-1})$. Then
\begin{equation}
\sum_{n,k,l\ge 0}\frac{t^{n+k+l}}{a^kc^l}
\frac{(aw^{\ppm};q)_k(cw^{\ppm};q)_l\,p_n(x;q^ka,b,q^lc,d|q)}
{(q;q)_k(q;q)_l(q;q)_n(ab;q)_{n+k}(cd;q)_{n+l}}
=\frac{(tw^{\ppm};q)_\infty}{(\frac{t}{a},\frac{t}{c};q)_\infty}.
\label{eqcor:1.3}
\end{equation}
\end{cor}
\begin{proof}
Simply setting $u=t$ in Theorem \ref{thm:32pf} produces the result.
\end{proof}

\begin{rem}
Another interesting consequence of Theorem \ref{thm:32pf} is obtained by
letting $w=d$ (equivalently $w=b$). In this case one of the ${}_3\phi_2$ basic hypergeometric series becomes unity. Further, in this case the Askey--Wilson polynomial simplifies using \cite[(130)]{KoornwinderKLSadd}
\begin{equation}\label{aw32}
p_n\big(\textstyle{\frac 12}(d+d^{-1});a,b,c,d|q\big)=d^{-n}(ad,bd,cd;q)_n.
\end{equation}
After this reduction, the sum over $l$ in \eqref{32pf} is summable
using the nonterminating
$q$-binomial theorem (cf.\ \cite[(II.3)]{GaspRah})
and the triple sum reduces to a double sum.
The final result is given by
\begin{eqnarray*}
&&\hspace{-2.5cm}\qhyp32{\frac{u}{t},ad,bd}{ab,du}{q,\frac{t}{d}}=\frac{(\frac{u}{a};q)_\infty}{(du;q)_\infty}
\sum_{n=0}^\infty\sum_{k=0}^\infty \frac{(\frac{a}{d};q)_k(bd;q)_n(ad;q)_{n+k}}{(q;q)_k(q;q)_n(ab;q)_{n+k}}
\left(\frac{t}{d}\right)^n\left(\frac{u}{a}\right)^k,
\end{eqnarray*}
which can be seen to be true since the right-hand side is a multiple of a $q$-Appell $\Phi^{(1)}$ function \cite[\href{http://dlmf.nist.gov/17.4.E5}{(17.4.5)}]{NIST:DLMF} which can be reduced to a ${}_3\phi_2$ using \cite[\href{http://dlmf.nist.gov/17.11.E1}{(17.11.1)}]{NIST:DLMF}.
\end{rem}


\section{Overview and revisitation of terminating $k$-balanced ${}_4\phi_3$ summations}\label{sec:4phi3}

Askey--Wilson polynomials are given in 
terms of a terminating balanced ${}_4\phi_3$. We would like to exploit the Ismail--Wilson generating function for Askey--Wilson polynomials
\eqref{AWgf} to obtain new product formulas for nonterminating basic hypergeometric series. (The triple summation in Theorem~\ref{thm:32pf}, while of
independent interest and having led to the discovery of Corollary~\ref{cor:1.3}, appears
to be not suitable for our purpose, due to the shifts of the parameters $a$
and $c$ of the Askey--Wilson polynomial that appears in the summand.)
There {exist} several well-known terminating ${}_4\phi_3$ summations. Some of them manifestly involve balanced summations
and others may be transformed to a balanced ${}_4\phi_3$ series
{and, as such, represent summations for Askey--Wilson polynomials which
can be applied}.
In the following, we inspect various known ${}_4\phi_3$ summation formulas to determine if they may be exploited as described above.

The summation formula \cite[\href{http://dlmf.nist.gov/17.7.E7}{(17.7.7)}]{NIST:DLMF}
\begin{equation*}
\qhyp43{a,-q\sqrt{a},b,c}{-\sqrt{a},\frac{qa}{b},\frac{qa}{c}}{q,\frac{q\sqrt{a}}{bc}}=
\frac{(qa,\frac{q\sqrt{a}}{b},\frac{q\sqrt{a}}{c},\frac{qa}{bc};q)_\infty}{(q\sqrt{a},\frac{qa}{b},\frac{qa}{c},\frac{q\sqrt{a}}{bc};q)_\infty},
\end{equation*} 
valid for $|q|<1$ and $|q\sqrt{a}|<|bc|$, concerns a summation for a nonterminating well-poised ${}_4\phi_3$ series (cf.\ \cite{GaspRah} for the notion of a well-poised series), but not one for a balanced ${}_4\phi_3$ series.
However, {the ${}_4\phi_3$ series} may be transformed to a terminating balanced sum by letting $c=q^{-n}$ and $b=q^n\sqrt{a}$. Unfortunately, in this case the formula vanishes for all $n\ge 1$ and is unity for $n=0$; so
the use of this summation will not lead to an interesting product formula {when} applied to the Ismail--Wilson generating function for Askey--Wilson polynomials \eqref{AWgf}.

In \cite{Andrews1976}, Andrews (1976) proved the following
balanced ${}_4\phi_3$ summation, hereby giving a
$q$-analogue of the terminating version of Watson's ${}_3F_2(1)$ summation \cite[\href{http://dlmf.nist.gov/16.4.E6}{(16.4.6)}]{NIST:DLMF}:
\begin{equation}\label{eq:Andrews}
\qhyp43{a,b,\ppm \sqrt{c}}{\ppm\sqrt{qab},c}{q,q}=a^{\frac n2}\frac{(qa,qb,\frac{qc}{a},\frac{qc}{b};q^2)_{\infty}}{(q,qab,qc,\frac{qc}{ab};q^2)_{\infty}},
\end{equation}
provided $b=q^{-n}$ and $n$ is a non-negative integer.
Replacing the variable $a$ by $q^na$ and written in explicit terms,
this formula is
\begin{equation}
\qhyp43{q^{-n},q^na,\ppm \sqrt{c}}{\ppm\sqrt{qa},c}{q,q}=
\begin{cases}
\frac{c^{\frac n2}(q,\frac{qa}{c};q^2)_{\frac n2}}{(qa,qc;q^2)_{\frac n2}}&\text{if $n$ even},\\
0&\text{if $n$ odd}.\end{cases}
\label{Andrews2}
\end{equation}
In \cite{Schlosser2018}, \eqref{Andrews2} was used to prove the product formula
\cite[Theorem~3]{Schlosser2018}
\begin{equation*}
\qhyp21{\ppm a}{a^2}{q,t}\qhyp21{\ppm b}{b^2}{q,-t}=\qhyp43{\ppm ab,\ppm qab}{qa^2,qb^2,a^2b^2}{q^2,t^2},
\end{equation*}
valid for $|q|<1$ and $|t|<1$.
By a standard polynomial argument, Andrews' summation in \eqref{Andrews2} is equivalent to the following summation found
by Jain~\cite{Jain1981}:
\begin{equation}\label{eq:Jain}
\qhyp43{\ppm q^{-n},a,b}{q^{-2n},\ppm\sqrt{qab}}{q,q}=\frac{(qa,qb;q^2)_n}{(q,qab;q^2)_n}.
\end{equation}
In fact, \eqref{eq:Andrews} holds whenever the series terminates;
i.e., instead of $b=q^{-n}$ we can take $c=q^{-2n}$, which yields \eqref{eq:Jain}. Now, while \eqref{eq:Jain} concerns a balanced $_4\phi_3$ summation, it cannot be written as an $n$-th order Askey--Wilson polynomial
(using any of the representations in Theorem~\ref{AWthm})
with parameters independent from $n$.

The very-well-poised ${}_8W_7$ appearing in the summation
\cite[\href{http://dlmf.nist.gov/17.7.E8}{(17.7.8)}]{NIST:DLMF}
(which is a Gasper and Rahman's $q$-analogue of Watson's $_3F_2$ sum) can be transformed to a sum of two balanced ${}_4\phi_3$'s using Bailey's transformation \cite[\href{http://dlmf.nist.gov/17.9.E16}{(17.9.16)}]{NIST:DLMF}.  In most choices of parameters which give a terminating representation for the ${}_8W_7$ one does not obtain a meaningful result as a single balanced ${}_4\phi_3$. However, one of {choices} leads to a balanced ${}_4\phi_3$ summation formula, namely
\begin{eqnarray}
&&\hspace{-0.7cm}\qhyp43{q^{-2n},c,-\frac{q^{1-n}}{b},\frac{q^{1-n}b}{c}}{\frac{q^{2-2n}}{c},-q^{1-n}b,\frac{q^{1-n}c}{b}}{q^2,q^2}
=
\frac{(q^{1-n}b,c^2,q^{2n}c,
q^{1+n}\frac{c}{b};q^2)_\infty(q^{2-2n},q^2b^2,q^{2n+2}c^2,\frac{q^2c^2}{b^2};q^4)_\infty}
{(q^{n+1}b,c,q^{2n}c^2,
q^{1-n}\frac{c}{b}
;q^2)_\infty(q^2,q^{2-2n}b^2,q^2c^2,q^{2n+2}\frac{c^2}{b^2};q^4)_\infty},
\label{newsum4}
\end{eqnarray}
which vanishes for $n$ odd. From
\eqref{newsum4} one obtains, after the substitution
$(b,c)\mapsto(-q^{1-n}/b,a)$, followed by replacing $q^2$ by $q$,
and some simplification, the following summation:
\begin{equation}
\qhyp43{q^{-n},-\frac{q^{1-n}}{ab},a,b}{-ab,\frac{q^{1-n}}{a},\frac{q^{1-n}}{b}}{q,q}=
\begin{cases}
\frac{(q,a^2,b^2;q^2)_{\frac n2}\,(ab;q)_{n}}{(a,b;q)_{n}(a^2b^2;q^2)_{\frac n2}}&
\text{if $n$ even,}\\
0&\text{if $n$ odd.}
\end{cases}
\label{newsum3}
\end{equation}
The summation formula \eqref{newsum3} is amenable to replacements in the Askey--Wilson polynomials using the terminating basic hypergeometric series representation \eqref{aw:def3}. This summation of a balanced terminating ${}_4\phi_3$ was first derived by Bailey (1941) \cite{Bailey1941} and is a consequence of a formula due to Jackson (1941) \cite{Jackson1941} (see also Carlitz (1969) \cite{Carlitz1969}, and \cite[Exercise 2.6]{GaspRah}).
Note that \eqref{newsum3} is equivalent to \eqref{Andrews2} (and \eqref{eq:Jain}), which can be seen by replacing $(a,c)\mapsto(a^2b^2/q,a^2)$ in \eqref{Andrews2} and then applying Sears' transformation 
\cite[\href{http://dlmf.nist.gov/17.9.E14}{(17.9.14)}]{NIST:DLMF}.
Therefore, it is seen that Andrews' $q$-analogue \eqref{Andrews2} of the terminating version of Watson's ${}_3F_2(1)$ summation \cite[\href{http://dlmf.nist.gov/16.4.E6}{(16.4.6)}]{NIST:DLMF} is a special case of Gasper--Rahman's nonterminating $q$-analogue of Watson's ${}_3F_2$ sum \cite[\href{http://dlmf.nist.gov/17.7.E8}{(17.7.8)}]{NIST:DLMF}.

The summation formula \cite[\href{http://dlmf.nist.gov/17.7.E10}{(17.7.10)}]{NIST:DLMF}
\begin{equation*}
\Whyp87{-c}{a,\frac{q}{a},c,-d,-\frac{q}{d}}{q,c}=\frac{(-c,-qc;q)_\infty(acd,\frac{qac}{d},\frac{qcd}{a},\frac{q^2c}{ad};q^2)_\infty}{(cd,\frac{qc}{d},-ac,-\frac{qc}{a};q)_\infty},
\end{equation*}
{where $|q|<1$ and $|c|<1$,}
Gasper {and} Rahman's $q$-analogue of Whipple's ${}_3F_2(1)$ summation \cite[\href{http://dlmf.nist.gov/16.4.E7}{(16.4.7)}]{NIST:DLMF}, is a very-well-poised ${}_8W_7$ summation formula. It can be transformed to a sum of two balanced ${}_4\phi_3$'s using Bailey's transformation \cite[\href{http://dlmf.nist.gov/17.9.E16}{(17.9.16)}]{NIST:DLMF}. If one chooses $(a,b,c,d,e,f)\mapsto(-c,a,q/a,c,-d,-q/d)$ with $d=q^{-n}$ then the coefficients of one of the balanced ${}_4\phi_3$'s vanishes and we are left with the following balanced ${}_4\phi_3$ summation formula
\begin{equation*}
\qhyp43{q^{-n},q^{n+1},\ppm c}{-q,-ac,-\frac{qc}{a}}{q,q}=
\frac{(-1,-q;q)_\infty(-q^{-n}ac,-q^{1+n}ac,-q^{1-n}\frac{c}{a},-q^{2+n}\frac{c}{a};q^2)_\infty}{(-q^{-n},-q^{n+1},-ac,-\frac{qc}{a};q)_\infty},
\end{equation*}
which is equivalent to \cite[\href{http://dlmf.nist.gov/17.7.E11}{(17.7.11)}]{NIST:DLMF}, Andrews' $q$-analogue of the terminating version of Whipple's ${}_3F_2(1)$ summation \cite[\href{http://dlmf.nist.gov/16.4.E7}{(16.4.7)}]{NIST:DLMF}, a balanced terminating ${}_4\phi_3$ summation. The simplified formula is \cite[Theorem 2]{Andrews1976}, \cite[\href{http://dlmf.nist.gov/17.7.E11}{(17.7.11)}]{NIST:DLMF},
\begin{equation}
\qhyp43{q^{-n},q^{n+1},\ppm c}{-q,e,\frac{qc^2}{e}}{q,q}=
q^{\binom{n+1}{2}}\frac{(q^{-n}e,q^{n+1}e,\frac{q^{1-n}c^2}{e},\frac{q^{n+2}c^2}{e};q^2)_\infty}{(e,\frac{qc^2}{e};q)_\infty},
\label{Andrews}
\end{equation}
also given in this form in \cite[(II.19)]{GaspRah},
and which can be more explicitly stated as
\begin{eqnarray}
&&\hspace{-0.60cm}\qhyp43{q^{-n},q^{n+1},\ppm a}{-q,b,\frac{qa^2}{b}}{q,q}
=\left\{ \begin{array}{ll}
\displaystyle \frac{a^n(\frac{q^2}{b},\frac{qb}{a^2};q^2)_{\frac{n}{2}}}
{(qb,\frac{q^2a^2}{b};q^2)_{\frac{n}{2}}} & \qquad\mathrm{if}\ n\ \mathrm{even},\\[15pt]
\displaystyle 
\frac{q(1-\frac{b}{q})(1-\frac{a^2}{b})}{(1-b)(1-\frac{qa^2}{b})}\frac{(-a)^{n-1}(\frac{q^3}{b},\frac{q^2b}{a^2};q^2)_{\frac{n-1}{2}}}{
(q^2b,\frac{q^3a^2}{b};q^2)_{\frac{n-1}{2}}}
& \qquad\mathrm{if}\ n\ \mathrm{odd}.
\end{array} \right.
\label{Andrews43}
\end{eqnarray}

The following two summations for balanced ${}_4\phi_3$ series are given in Bressoud, Ismail and Stanton (2000) \cite[(2.1), (2.2)]{BressoudIsmailStanton}. However, they are given even earlier in Verma and Jain (1980) \cite[(5.3), (5.4)]{VermaJain1980}.
The first balanced ${}_4\phi_3$ summation formula  \cite[(2.2)]{BressoudIsmailStanton}, \cite[\href{http://dlmf.nist.gov/17.7.E12}{(17.7.12)}]{NIST:DLMF}
(also see \cite[Exercise~3.34; $c\mapsto bq^n$]{GaspRah}) is given by 
\begin{equation}\label{eq:BIS-2.2}
\qhyp43{q^{-2n},q^{2n}b^2,a,qa}{b,qb,q^2a^2}{q^2,q^2}=\frac{a^n(-q,\frac{b}{a};q)_n}{(-qa,b;q)_n},
\end{equation}
which is a $q$-analogue of Bailey's ${}_4F_3(1)$ summation. We will use this summation formula {to derive} Theorem \ref{thm21} in Section~\ref{sec:npt}. The balanced ${}_4\phi_3$ summation formula \cite[\href{http://dlmf.nist.gov/17.7.E13}{(17.7.13)}]{NIST:DLMF}, \cite[(2.1)]{BressoudIsmailStanton}
\begin{equation}\label{eq:BIS-2.1}
\qhyp43{q^{-2n},q^{2n}b^2,a,qa}{qb,q^2b,a^2}{q^2,q^2}=\frac{a^n(1-b)(-q,\frac{qb}{a};q)_n}{(1-q^{2n}b)(-a,b;q)_n},
\end{equation}
is another $q$-analogue of Bailey's ${}_4F_3(1)$ summation. We will use \eqref{eq:BIS-2.1} to recover a result by Nassrallah \eqref{NassrallahB} (see also Remark \ref{rem34}). 
Guo (2013) \cite{G2013} and Wei--Wang (2015) \cite{WX2013} derived terminating balanced ${}_4\phi_3$ summations. 
Guo's balanced summations are 
\cite[p.~1040, second identity]{G2013}
\begin{equation}\label{eq:g}
\qhyp43{q^{-2n},a,b,\frac{q^{3-2n}}{ab}}{\frac{q^{2-2n}}a,\frac{q^{2-2n}}b,qab}{q^2,q^2}=
\frac{(-q,a,b;q)_n(ab;q^2)_n}{(1-q^{2n-1}ab)\,(ab;q)_{n-1}(a,b;q^2)_n},
\end{equation}
\cite[(4.4)]{G2013}
\begin{equation}\label{eq:g2}
\qhyp43{q^{-2n},q^2a,q^2b,\frac{q^{1-2n}}{ab}}{\frac{q^{2-2n}}a,\frac{q^{2-2n}}b,q^3ab}{q^2,q^2}=
\frac{q^{-n}(-q,qa,qb;q)_n(q^2ab;q^2)_n}{(1-q^{2n+1}ab)\,(q^2ab;q)_{n-1}(a,b;q^2)_n},
\end{equation}
and \cite[(4.6)]{G2013}
\begin{equation}\label{eq:g3}
\qhyp43{q^{-2n},a,b,\frac{q^{-1-2n}}{ab}}{\frac{q^{-2n}}a,\frac{q^{-2n}}b,qab}{q^2,q^2}=
\frac{(-q,qa,qb;q)_n(q^2ab;q^2)_n}{(qab;q)_{n}(q^2a,q^2b;q^2)_n}.
\end{equation}
Wei and Wang provided the following balanced summation \cite[Corollary 6]{WX2013}
\begin{equation}\label{eq:wx}
\qhyp43{q^{-2n},a,b,\frac{q^{3-2n}}{ab}}{\frac{q^{2-2n}}a,\frac{q^{4-2n}}b,\frac{ab}{q}}{q^2,q^2}=
\frac{q^{-n}(-q,a,\frac{b}{q};q)_n(\frac{ab}{q^2};q^2)_n}{(\frac{ab}{q^2};q)_{n}(a,\frac{b}{q^2};q^2)_n},
\end{equation}
and a number of closely related ones. 

\begin{rem}
\label{Rem31}
On first appearance, one might consider \eqref{eq:BIS-2.2}--\eqref{eq:wx} to be different. However, after replacing $q^2\mapsto q$, all the summations \eqref{eq:BIS-2.2}--\eqref{eq:wx} should be considered to be equivalent when viewed as terminating balanced ${}_4\phi_3$'s which can be written in terms of Askey--Wilson polynomials, namely with Theorem \ref{thmWeiWang} below.
Using \eqref{aw:def1}, the summations \eqref{eq:BIS-2.2} and \eqref{eq:BIS-2.1} produce Theorem \ref{thmWeiWang} by replacing $(a,b)\mapsto(q^{-\frac14}a,ab)$ and $(a,b)\mapsto(q^\frac14a,q^\frac12ab)$ respectively.
Using \eqref{aw:def3}, the summations \eqref{eq:g}--%
\eqref{eq:wx} 
produce Theorem \ref{thmWeiWang} by replacing $(a,b)\mapsto (q^{-\frac14}a,q^{-\frac14}b)$, $(a,b)\mapsto(q^{\frac14}a,q^{\frac14}a)$, and $(a,b)\mapsto (q^{-\frac14}a,q^{-\frac14}b)$, $(a,b)\mapsto(q^{\frac14}a,q^{\frac34}b)$ respectively.
\end{rem}


The summation \cite[\href{http://dlmf.nist.gov/17.7.E14}{(17.7.14)}]{NIST:DLMF}, F.~H.~Jackson's $q$-analogue of Dougall's ${}_7F_6(1)$ summation, is a summation of a {terminating} very-well-poised balanced ${}_8W_7$. It can be written as a multiple of a balanced ${}_4\phi_3$ by taking $(a,b,c,d,e,f)\mapsto(a,b,c,d,\frac{q^{n+1}a^2}{bcd},q^{-n})$ in Bailey's transformation \cite[\href{http://dlmf.nist.gov/17.9.E16}{(17.9.15)}]{NIST:DLMF}. After simplification {this} reduces to
\begin{equation*}
\qhyp32{q^{-n},\frac{q^{n+1}a^2}{bcd},d}{\frac{qa}{b},\frac{qa}{c}}{q,q}=d^n
\frac{(\frac{qa}{bd},\frac{qa}{cd};q)_n}{(\frac{qa}{b},\frac{qa}{c};q)_n},
\end{equation*}
{which is equivalent to the $q$-Pfaff--Saalsch\"utz sum \cite[\href{http://dlmf.nist.gov/17.7.E4}{(17.7.4)}]{NIST:DLMF}},
namely
\begin{equation}\label{eq:3ph2}
\qhyp32{q^{-n},a,b}{c,\frac{q^{1-n}ab}{c}}{q,q}=
\frac{(\frac{c}{a},\frac{c}{b};q)_n}{(c,\frac{c}{ab};q)_n}.
\end{equation}
This is exactly the summation which is responsible for the evaluation
in \eqref{aw32}, and as such useful for reducing 
the Ismail--Wilson generating function for Askey--Wilson polynomials
by a suitable specialization of the parameters.

The summation \cite[Exercise 2.14 (i); $a\mapsto a^2$]{GaspRah}
\begin{equation}\label{GR-2.14-i}
\qhyp43{q^{-n},b,a^2,qa}{b^2q^{1-n},\frac{qa^2}{b},a}{q,q}=\frac{(1+\frac abq^n)(\frac{a^2}{b^2},\frac1{b};q)_n}{(1+\frac ab)(\frac{qa^2}b,\frac1{b^2};q)_n}
\end{equation}
is indeed a summation for a balanced ${}_4\phi_3$ series but in order to evaluate
an $n$-th order Askey--Wilson polynomial (represented as a ${}_4\phi_3$ series using one of the representations in Theorem~\ref{AWthm}) using \eqref{GR-2.14-i} would require to
make one or more of the parameters $a,b,c,d,w$ to depend on $n$ (which can not be conveniently done when considering
generating functions).

Two further terminating balanced ${}_4\phi_3$ summations are listed in Gasper and Rahman's textbook.
There is \cite[(3.10.9); $a\mapsto a^2$ and $w\mapsto abq^{1-n}$]{GaspRah}
\begin{equation}\label{GR-3.10.9}
\qhyp43{q^{-n},-bq^{-n},a^2,qa}{abq^{1-n},-aq^{1-n},a}{q,q}=(qa^2)^{-n}\frac{(1-\frac abq^{2n})(\frac{qa}{b},-a;q)_n}{(1-\frac abq^n)(\frac1{ab},-\frac1{a};q)_n},
\end{equation}
and also \cite[(3.10.10); $a\mapsto ab$]{GaspRah}
\begin{equation}\label{GR-3.10.10}
\qhyp43{q^{-n},-bq^{1-n},ab,b}{b^2q^{1-n},-bq^{-n},qa}{q,q}=\frac{(1+\frac 1b)(1-\frac abq^{2n})(\frac{a}{b},\frac 1b;q)_n}{(1+\frac{q^n}b)(1-\frac ab)(aq,\frac1{b^2};q)_n}.
\end{equation}
The last three summations can actually be easily proved by application of the Sears transformation
\cite[(3.2.1)]{GaspRah} after which the transformed sum, having $q^{-1}$ as an upper parameter,
reduces to the sum of its first two terms.
The transformation \eqref{GR-3.10.10} can be seen to be equivalent to \eqref{GR-3.10.9} by first interchanging $a\leftrightarrow b$, replacing $a\mapsto qa$ followed by $(a,b)\mapsto(1/(qa),1/b)$ and then replacing $z\mapsto 1/z$.

Further, there are \cite[(3.7)]{VJ1983}
\begin{equation}\label{eq:bws}
\qhyp43{q^{-n},q^{1-n},a^2,b^2}{q^{2-2n},ab,qab}{q^2,q^2}=\frac{(-a,b;q)_n+(a,-b;q)_n}{(-1,ab;q)_n},
\end{equation}
and \cite[(3.8)]{VJ1983}
\begin{equation}\label{eq:bws2}
\qhyp43{q^{-n},q^{1-n},a^2,b^2}{q^{-2n},qab,q^2ab}{q^2,q^2}=\frac{(-a,b;q)_{n+1}-(a,-b;q)_{n+1}}{(a-b)\,(-1;q)_{n+1}(qab;q)_n},
\end{equation}
which cannot be written as $n$-th order Askey--Wilson polynomials with parameters
independent from $n$. The summation in \eqref{eq:bws} can be obtained by a careful ($c\to1$) limit from
the following ${}_4\phi_3$ transformation of Berkovich and Warnaar \cite[(3.10)]{BerkovichWarnaar2005}:
\begin{equation}\label{eq:bwt}
\qhyp43{q^{-n},b, c,-c}{-\frac{q^{1-n}b}a,a,c^2}{q,q}=\frac{(\frac{a^2}b;q)_n(c^2;q^2)_n}{(-\frac ab,a,c^2;q)_n}
\qhyp43{q^{-n},q^{1-n},\frac{a^2}{b^2},\frac{a^2}{c^2}}{\frac{q^{2-2n}}{c^2},\frac{a^2}b,\frac{qa^2}b}{q^2,q^2}.
\end{equation}
Specifically, the limit $c\to 1$ in \eqref{eq:bwt} using
\begin{eqnarray*}
&&\hspace{-8cm}\lim_{c\to1}\frac{(c^2;q^2)_k}{(c^2;q)_k}=\begin{cases}1&\text{if $k=0$,}\\\frac 12(-1;q)_k\qquad&\text{if $k>0$},\end{cases}
\end{eqnarray*}
reduces the sum on the left-hand side of the transformation which can then be computed by the
$q$-Pfaff--Saalsch\"utz summation \cite[\href{http://dlmf.nist.gov/17.7.E4}{(17.7.4)}]{NIST:DLMF},
boiling down, after replacing $b$ by $a/b$, to \eqref{eq:bws}. (This derivation of \eqref{eq:bws} from \eqref{eq:bwt} was kindly communicated to us by S.~Ole Warnaar in private communication).

\medskip
Different types of balanced and 2-balanced summations were derived by Andrews (2011)~\cite{A2011} and by Guo (2013) \cite{G2013}.
The involved sums have the special feature that they are symmetric or almost symmetric with respect to reversing them.
We list a few of them as samples.
In \cite[Theorem~3]{A2011} Andrews proved the following 2-balanced ${}_4\phi_3$ summation.
\begin{equation}
\qhyp43{q^{-2n},a,b,\frac{q^{1-2n}}{ab}}{\frac{q^{2-2n}}a,\frac{q^{2-2n}}b,qab}{q^2,q^2}=
\frac{q^{-n}(a,b,-q;q)_n(ab;q^2)_n}{(ab;q)_n(a,b;q^2)_n}.
\end{equation}
Andrews further made the following related conjecture \cite[(6.1)]{A2011} (again concerning
a 2-balanced summation) that was settled by Guo~\cite[Section~3]{G2013}
\begin{eqnarray}
&&\hspace{-0.85cm}\qhyp43{q^{-2n},a,b,\frac{q^{3-2n}}{ab}}{\frac{q^{2-2n}}a,\frac{q^{4-2n}}b,qab}{q^2,q^2}\notag\nonumber\\
&&=\frac{-q^{-1-n}(a,-q;q)_n(b;q)_{n-1}(ab;q^2)_{n-1}}
{(1-q^{2n-1}ab)\,(ab;q)_{n-1}(a,\frac b{q^2};q^2)_n}(q^{2n-2}ab(q^2-b)+q^{n-1}ab(1-q)-q+b).
\label{IsmailR52}
\end{eqnarray}
Both of the above 2-balanced ${}_4\phi_3$ summations conjectured by Andrews were also proved in Ismail and Rahman (2011) \cite{IsmailRahman11}, but the statement and proof of \eqref{IsmailR52} therein both contain typographical errors.

\subsection{New terminating $1$-balanced ${}_4\phi_3$ summations}\label{subsec:kbal}

\noindent While carrying out a systematic search in the literature for terminating balanced ${}_4\phi_3$ summations,
and making computer experiments, we found the following additional summations that are similar to \eqref{Andrews2},
which appear be new.
First we consider summations of terminating ($1$-)balanced ${}_4\phi_3$ series.
\medskip

\begin{thm}
\label{thm31}
Let $q,a,c\in\CCast$. Then one has the following two equivalent $q$-quadratic summations for a terminating balanced ${}_4\phi_3$, namely 
\begin{equation}
\qhyp43{q^{-n},q^na,\ppm \sqrt{c}}{qc,\ppm \sqrt{a}}{q,q}=c^{\lfloor\frac{n+1}2\rfloor}
\frac{(q,\frac{a}{c};q^2)_{\lfloor\frac{n+1}2\rfloor}}{(a,qc;q^2)_{\lfloor\frac{n+1}2\rfloor}},
\label{eq:n2}
\end{equation}
\begin{equation}
\qhyp43{q^{-n},q^na,\ppm q\sqrt{c}}{qc,\ppm q\sqrt{a}}{q,q}=(-q)^nc^{\lfloor\frac{n+1}2\rfloor}
\frac{(1-a)}{(1-q^{2n}a)}\frac{(q,\frac{a}{c};q^2)_{\lfloor\frac{n+1}2\rfloor}}{(a,qc;q^2)_{\lfloor\frac{n+1}2\rfloor}},
\label{eq:n1}
\end{equation}
where the two summations in \eqref{eq:n2} and \eqref{eq:n1} can be shown
to be equivalent using 
Sears' balanced terminating ${}_4\phi_3$ transformation 
\cite[\href{http://dlmf.nist.gov/17.9.E14}{(17.9.14)}]{NIST:DLMF}.
\end{thm}
\begin{proof}
We only need to show \eqref{eq:n2} which we do by induction on $n$.
For $n=0$ the summation in \eqref{eq:n2} is trivial. We assume that the formula is true up to a fixed non-negative integer $n$. To prove \eqref{eq:n2} for $n$ replaced by $n+1$, we use the elementary identity
\begin{equation}\label{eq:elid}
\frac{(1-q^{-n-1})(1-q^{n+k}a)}{(1-q^{-n-1+k})(1-q^{n}a)}
=1-q^{-n-1}\frac{(1-q^k)(1-q^{2n+1}a)}{(1-q^{-n-1+k})(1-q^{n}a)}
\end{equation}
to obtain
\begin{align*}
&\qhyp43{q^{-n-1},q^{n+1}a,\ppm \sqrt{c}}{qc,\ppm \sqrt{a}}{q,q}\\
&=
\qhyp43{q^{-n},q^{n}a,\ppm \sqrt{c}}{qc,\ppm \sqrt{a}}{q,q}
-q^{-n}\frac{(1-c)(1-q^{2n+1}a)}{(1-qc)(1-a)}
\qhyp43{q^{-n},q^{n+1}a,\ppm q\sqrt{c}}{q^2c,\ppm q\sqrt{a}}{q,q}.
\end{align*}
The first $_4\phi_3$ series can be simplified by the inductive
hypothesis, whereas the second series can be evaluated by \eqref{Andrews2},
and, in particular, vanishes for odd $n$. As a result, we obtain
for the last expression,
\begin{eqnarray*}
&&\hspace{-11.5cm}c^{\lfloor\frac{n+1}2\rfloor}
\frac{(q,\frac{a}{c};q^2)_{\lfloor\frac{n+1}2\rfloor}}{(a,qc;q^2)_{\lfloor\frac{n+1}2\rfloor}}
\end{eqnarray*}
in case $n$ is odd, and
\begin{align*}
&c^{\lfloor\frac{n+1}2\rfloor}
\frac{(q,\frac{a}{c};q^2)_{\lfloor\frac{n+1}2\rfloor}}{(a,qc;q^2)_{\lfloor\frac{n+1}2\rfloor}}-
q^{-n}\frac{(1-c)(1-q^{2n+1}a)}{(1-qc)(1-a)}
c^{\lfloor\frac{n+1}2\rfloor}
\frac{(q,\frac{a}{c};q^2)_{\lfloor\frac{n+1}2\rfloor}}{(q^2a,q^3c;q^2)_{\lfloor\frac{n+1}2\rfloor}}\\
&=[(1-q^{n+1}c)(1-q^na)-(1-c)(1-q^{2n+1}a)]
c^{\lfloor\frac{n+1}2\rfloor}
\frac{(q,\frac{a}{c};q^2)_{\lfloor\frac{n+1}2\rfloor}}{(a,qc;q^2)_{\lfloor\frac{n+1}2+1\rfloor}}\\
&=c(1-q^{n+1})(1-\frac{q^na}c)c^{\lfloor\frac{n+1}2\rfloor}
\frac{(q,\frac{a}{c};q^2)_{\lfloor\frac{n+1}2\rfloor}}{(a,qc;q^2)_{\lfloor\frac{n+1}2+1\rfloor}}
=c^{\lfloor\frac{n+1}2+1\rfloor}
\frac{(q,\frac{a}{c};q^2)_{\lfloor\frac{n+1}2+1\rfloor}}{(a,qc;q^2)_{\lfloor\frac{n+1}2+1\rfloor}}
\end{align*}
in case $n$ is even, which furnishes the theorem.
\end{proof}

\begin{thm}
\label{thmn5}
Let $q,a,c\in\CCast$. Then one has the following $q$-quadratic summation for a terminating balanced ${}_4\phi_3$, namely 
\begin{eqnarray}
&&\hspace{-1.5cm}\qhyp43{q^{-n},q^{n}a,\ppm\sqrt{c}}{c,q\sqrt{a},-\sqrt{a}}{q,q}=
c^{\lfloor\frac{n}2\rfloor}a^{\frac n2-\lfloor\frac{n}2\rfloor}\frac{(1-\sqrt{a})}{(1-q^n\sqrt{a})}
\frac{(q;q^2)_{\lfloor\frac{n+1}2\rfloor}}{(a;q^2)_{\lfloor\frac{n+1}2\rfloor}}
\frac{(\frac{q^2a}{c};q^2)_{\lfloor\frac{n}2\rfloor}}{(qc;q^2)_{\lfloor\frac{n}2\rfloor}}.
\label{eq:nn5}
\end{eqnarray}
\end{thm}
\begin{proof}
This result readily follows from taking $(a,b,c)\mapsto(-\sqrt{a},q^{n}a\sqrt{c})$ in \eqref{eq:bwt}.
In this case the $_4\phi_3$ series on the right-hand side can be seen to reduce to the series
\begin{equation*}
\qhyp32{q^{-\lfloor\frac{n}2\rfloor},\frac{q^{-2n}}{a},\frac{a}{c}}{\frac{q^{2-2n}}{c},q^{-\lfloor\frac{n}2\rfloor}}{q^2,q^2},
\end{equation*}
which factorizes in closed form due to the $q$-Pfaff--Saalsch\"utz summation \eqref{eq:3ph2}.
Simplification of the combined factors gives \eqref{eq:nn5}.
\end{proof}

\noindent The following terminating balanced ${}_4\phi_3$ is not new but due to Wei, Gong and Li \cite[Corollary~5]{WGL2013},
nevertheless we display it in this section, together with a direct proof.
\begin{thm}
\label{thm33new}
Let $q,a,c\in\CCast$. Then one has the following $q$-quadratic summation for a terminating balanced ${}_4\phi_3$, namely 
\begin{eqnarray}
&&\hspace{-3.5cm}\qhyp43{q^{-n},q^{n+1}a,q\sqrt{c},-\sqrt{c}}{qc,\ppm q\sqrt{a}}{q,q}=
(-1)^n c^{\frac n2}
\frac{(q;q^2)_{\lfloor\frac{n+1}2\rfloor}}{(qc;q^2)_{\lfloor\frac{n+1}2\rfloor}}
\frac{(\frac{q^2a}{c};q^2)_{\lfloor\frac{n}2\rfloor}}{(q^2a;q^2)_{\lfloor\frac{n}2\rfloor}}.
\label{eq:nn3}
\end{eqnarray}
\end{thm}
\begin{proof}
This immediately follows from the $(\sqrt{a},\sqrt{c})\mapsto(-\frac{q^{-n}}{\sqrt{c}},-\frac{q^{-n}}{\sqrt{a}})$
substitution of Theorem~\ref{thmn5}, together with reversing the order of summation of the
$_4\phi_3$ series according to \cite[Exercise 1.4 (ii)]{GaspRah}.
\end{proof}

\begin{thm}
\label{thm34new}
Let $q,a,c\in\CCast$. Then one has the following two equivalent $q$-quadratic summations for a terminating balanced ${}_4\phi_3$, namely 
\begin{eqnarray}
&&\hspace{-0.5cm}\qhyp43{q^{-n},q^{n+1}a,\ppm\sqrt{c}}{qc,q\sqrt{a},-\sqrt{a}}{q,q}=
c^{\lfloor\frac n2\rfloor}
\frac{(q;q^2)_{\lfloor\frac{n+1}2\rfloor}}{(qc;q^2)_{\lfloor\frac{n+1}2\rfloor}}
\frac{(\frac{q^2a}{c};q^2)_{\lfloor\frac{n}2\rfloor}}{(q^2a;q^2)_{\lfloor\frac{n}2\rfloor}}\notag\\
&&\hspace{5.5cm}
\times
\left\{ \begin{array}{ll}
\displaystyle 
1
& \ \mathrm{if}\ n\ \mathrm{even},\\[4pt]
\displaystyle \frac{(c+\sqrt{a})}{(1+\sqrt{a})} & \ \mathrm{if}\ n\ \mathrm{odd},
\end{array} \right.
\label{eq:nn2}
\\
&&\hspace{-0.5cm}\qhyp43{q^{-n},q^{n+1}a,\ppm q\sqrt{c}}{qc,q\sqrt{a},-q^2\sqrt{a}}{q,q}=
(-1)^n q^n c^{\lfloor\frac n2\rfloor}\frac{(1+\sqrt{a})(1+q\sqrt{a})}{(1+q^{n}\sqrt{a})(1+q^{n+1}\sqrt{a})}
\frac{(q;q^2)_{\lfloor\frac{n+1}2\rfloor}}{(qc;q^2)_{\lfloor\frac{n+1}2\rfloor}}
\frac{(\frac{q^2a}{c};q^2)_{\lfloor\frac{n}2\rfloor}}{(q^2a;q^2)_{\lfloor\frac{n}2\rfloor}}\notag\\
&&\hspace{5.5cm}
\times
\left\{ \begin{array}{ll}
\displaystyle 1
& \ \mathrm{if}\ n\ \mathrm{even},\\[4pt]
\displaystyle \frac{(c+\sqrt{a})}{(1+\sqrt{a})} & \ \mathrm{if}\ n\ \mathrm{odd},
\end{array} \right.
\label{eq:nn2r}
\end{eqnarray}
where the two summations in \eqref{eq:nn2} and \eqref{eq:nn2r} can be shown
to be equivalent using 
Sears' balanced terminating ${}_4\phi_3$ transformation 
\cite[\href{http://dlmf.nist.gov/17.9.E14}{(17.9.14)}]{NIST:DLMF}.
\end{thm}
\begin{proof}
We only need to show \eqref{eq:nn2} which we achieve by the use of contiguous relations.
In particular, application of \cite[(3.3)]{Krattenthaler1993}
(which is a straightforward extension of \cite[Exercise 1.9 (iii)]{GaspRah})
yields
\begin{multline*}
\qhyp43{q^{-n},q^{n+1}a,\ppm\sqrt{c}}{qc,q\sqrt{a},-\sqrt{a}}{q,q}=
\qhyp43{q^{-n},q^{n}a,\ppm\sqrt{c}}{c,q\sqrt{a},-\sqrt{a}}{q,q}\\+
\frac{q(q^na-c)(1-q^{-n})}{(1-qc)(1-q\sqrt{a})(1+\sqrt{a})}
\qhyp43{q^{1-n},q^{n+1}a,\ppm q\sqrt{c}}{q^2c,q^2\sqrt{a},-q\sqrt{a}}{q,q}.
\end{multline*}
Now the two $_4\phi_3$ series on the right-hand side can be
evaluated by different instances of \eqref{eq:nn5}.
The two terms can be combined to the right-hand side in \eqref{eq:nn2}.
\end{proof}

\begin{thm}
\label{thm35new}
Let $q,a,c\in\CCast$. Then one has the following $q$-quadratic summation for a terminating balanced ${}_4\phi_3$, namely 
\begin{eqnarray}
&&\hspace{-2.6cm}\qhyp43{q^{-n},q^{n}a,q\sqrt{c},-\sqrt{c}}{qc,q\sqrt{a},-\sqrt{a}}{q,q}=
(-1)^n\frac{c^{\lfloor\frac{n+1}2\rfloor}}{(\sqrt{c}+\sqrt{a})}\frac{(1-\sqrt{a})}{(1-q^n\sqrt{a})}
\frac{(q,\frac{a}{c};q^2)_{\lfloor\frac{n+1}2\rfloor}}{(a,qc;q^2)_{\lfloor\frac{n+1}2\rfloor}}\notag\\
&&\hspace{3.8cm}
\times
\left\{ \begin{array}{ll}
\displaystyle 
(\sqrt{c}+q^n\sqrt{a})
& \ \mathrm{if}\ n\ \mathrm{even},\\[6pt]
\displaystyle (1+q^n\sqrt{ac}) & \ \mathrm{if}\ n\ \mathrm{odd}.
\end{array} \right.
\label{eq:nn1}
\end{eqnarray}
\end{thm}
\begin{proof}
Just as in the proof of Theorem~\ref{thm34new}, we use the
contiguous relation \cite[(3.3)]{Krattenthaler1993}.
This yields
\begin{multline*}
\qhyp43{q^{-n},q^{n}a,q\sqrt{c},-\sqrt{c}}{qc,q\sqrt{a},-\sqrt{a}}{q,q}=
\qhyp43{q^{-n},q^{n}a,\ppm\sqrt{c}}{qc,\ppm\sqrt{a}}{q,q}\\
+\frac{q(\sqrt{c}-\sqrt{a})(1-q^{-n})(1-q^{n}a)(1+\sqrt{c})}{(1-qc)(1-q\sqrt{a})(1-a)}
\qhyp43{q^{1-n},q^{n+1}a,\ppm q\sqrt{c}}{q^2c,q^2\sqrt{a},-q\sqrt{a}}{q,q}.
\end{multline*}
Now the two $_4\phi_3$ series on the right-hand side can be
evaluated by \eqref{eq:n2} and \eqref{eq:nn5}, respectively.
The two terms can be combined to the right-hand side in \eqref{eq:nn1}.
\end{proof}

\begin{thm}
\label{thm36new}
Let $q,a,c\in\CCast$. Then one has the following two equivalent $q$-quadratic summations for a terminating balanced ${}_4\phi_3$, namely 
\begin{eqnarray}
&&\hspace{-3.0cm}\qhyp43{q^{-n},q^{n}a,q\sqrt{c},-\sqrt{c}}{q^2c,\ppm\sqrt{a}}{q,q}=
(-1)^n c^{\lfloor\frac{n}2\rfloor}\frac{(1-qc)}{(1+q\sqrt{c})}
\frac{(q;q^2)_{\lfloor\frac{n+1}2\rfloor}(\frac{a}{c};q^2)_{\lfloor\frac{n}2\rfloor}}{(a,qc;q^2)_{\lfloor\frac{n+1}2\rfloor}}\notag\\
&&\hspace{3.2cm}
\times
\left\{ \begin{array}{ll}
\displaystyle 
\frac{(1+q^{n+1}\sqrt{c})}{(1-q^{n+1}c)}
& \ \mathrm{if}\ n\ \mathrm{even},\\[12pt]
\displaystyle (\sqrt{c}+q^{n-1}a) & \ \mathrm{if}\ n\ \mathrm{odd},
\end{array} \right.
\label{eq:nn4}
\end{eqnarray}
\begin{eqnarray}
&&\hspace{-1.5cm}\qhyp43{q^{-n},q^{n}a,q\sqrt{c},-q^2\sqrt{c}}{q^2c,\ppm q\sqrt{a}}{q,q}=
q^n c^{\lfloor\frac{n}2\rfloor}
\frac{(1-a)}{(1-q^{2n}a)}\frac{(1-qc)}{(1+q\sqrt{c})}
\frac{(q;q^2)_{\lfloor\frac{n+1}2\rfloor}(\frac{a}{c};q^2)_{\lfloor\frac{n}2\rfloor}}{(a,qc;q^2)_{\lfloor\frac{n+1}2\rfloor}}\notag\\
&&\hspace{4.75cm}
\times
\left\{ \begin{array}{ll}
\displaystyle 
\frac{(1+q^{n+1}\sqrt{c})}{(1-q^{n+1}c)}
& \ \mathrm{if}\ n\ \mathrm{even},\\[12pt]
\displaystyle (\sqrt{c}+q^{n-1}a) & \ \mathrm{if}\ n\ \mathrm{odd},
\end{array} \right.
\label{eq:nn4r}
\end{eqnarray}
where the two summations in \eqref{eq:nn4} and \eqref{eq:nn4r} can be shown
to be equivalent using 
Sears' balanced terminating ${}_4\phi_3$ transformation 
\cite[\href{http://dlmf.nist.gov/17.9.E14}{(17.9.14)}]{NIST:DLMF}.
\end{thm}
\begin{proof}
It suffices to prove \eqref{eq:nn4}.
Just as in the proofs of Theorems~\ref{thm34new} and \ref{thm35new},
we use the contiguous relation \cite[(3.3)]{Krattenthaler1993}.
This yields
\begin{eqnarray*}
&&\hspace{-0.3cm}\qhyp43{q^{-n},q^{n}a,q\sqrt{c},-\sqrt{c}}{q^2c,\ppm\sqrt{a}}{q,q}\\[0.05cm]
&&\hspace{0.7cm}=
\qhyp43{q^{-n},q^{n}a,\ppm\sqrt{c}}{qc,\ppm\sqrt{a}}{q,q}+\frac{q \sqrt{c}\,(1-q^{-n})(1-q^{n}a)(1+\sqrt{c})}{(1+q\sqrt{c})(1-qc)(1-a)}
\qhyp43{q^{1-n},q^{n+1}a,\ppm q\sqrt{c}}{q^3c,\ppm q\sqrt{a}}{q,q}.
\end{eqnarray*}
Now the two $_4\phi_3$ series on the right-hand side can be
evaluated by different instances of \eqref{eq:n2}.
The two terms can be combined to the right-hand side in \eqref{eq:nn4}.
\end{proof}

\noindent The next balanced ${}_4\phi_3$ summation does not factorize completely
(in particular not for even $n$), but possesses a complete product evaluation for odd $n$.
\begin{thm}
\label{thm33}
Let $q,a,c\in\CCast$. Then one has the following two equivalent $q$-quadratic summations for a terminating balanced ${}_4\phi_3$, namely
\begin{eqnarray}
&&\hspace{-0.5cm}\qhyp43{q^{-n},q^{n+1}a,\ppm \sqrt{c}}{q^2c,\ppm \sqrt{a}}{q,q}=
\frac{c^{\lfloor\frac{n+1}2\rfloor}}{(1-q^2c)}
\frac{(q,\frac ac;q^2)_{\lfloor\frac{n+1}2\rfloor}}{(a;q^2)_{\lfloor\frac{n+2}2\rfloor}(q^3c;q^2)_{\lfloor\frac n2\rfloor}}\notag\\
&&\hspace{5.0cm}
\times
\left\{ \begin{array}{ll}
\displaystyle 
(qc-q^na)(1-q^{n+1})+(1-q^{n+1}a)(1-q^{n+1}c)
& \ \mathrm{if}\ n\ \mathrm{even},\\[10pt]
\displaystyle (1+q) & \ \mathrm{if}\ n\ \mathrm{odd},
\end{array} \right.
\label{eq:n5}
\end{eqnarray}
\begin{eqnarray}
&&\hspace{-0.3cm}\qhyp43{q^{-n},q^{n+1}a,\ppm q^2\sqrt{c}}{q^2c,\ppm q^2\sqrt{a}}{q,q}=
\frac{(-1)^n q^{2n}c^{\lfloor\frac{n+1}2\rfloor}\,(1-q^2 a)}{(1-q^2c)(1-q^{2n}a)(1-q^{2n+2}a)}
\frac{(q,\frac ac;q^2)_{\lfloor\frac{n+1}2\rfloor}}{(q^2 a,q^3c;q^2)_{\lfloor\frac n2\rfloor}}\notag\\
&&\hspace{5.3cm}
\times
\left\{ \begin{array}{ll}
\displaystyle 
(qc-q^na)(1-q^{n+1})+(1-q^{n+1}a)(1-q^{n+1}c)
& \ \mathrm{if}\ n\ \mathrm{even},\\[10pt]
\displaystyle (1+q) & \ \mathrm{if}\ n\ \mathrm{odd},
\end{array} \right.
\label{eq:n5r}
\end{eqnarray}
where the two summations in \eqref{eq:n5} and \eqref{eq:n5r} can be shown
to be equivalent using 
Sears' balanced terminating ${}_4\phi_3$ transformation 
\cite[\href{http://dlmf.nist.gov/17.9.E14}{(17.9.14)}]{NIST:DLMF}.
\end{thm}

\begin{proof}
The proof is similar to that of Theorem~\ref{thm31} and uses induction on $n$.
It suffices to prove \eqref{eq:n5}.
Already the $n=0$ is more involved than that of Theorem~\ref{thm31} but still trivial.
Using the $a\mapsto qa$ case of \eqref{eq:elid}, the $n\mapsto n+1$ case of the sum
is split into two terms. The first term is simplified by the inductive hypothesis
(where division into two cases has to be done, according to the parity of $n$)
while the second term is simplified by an instance of Theorem~\ref{thm31}.
The computational details are routine and thus omitted.
\end{proof}

\noindent 
The next two balanced ${}_4\phi_3$ summations are similar in nature to Theorem~\ref{thm33}; they possess
complete product evaluations for odd $n$.
\begin{thm}
\label{thmf1}
Let $q,a,c\in\CCast$. Then one has the following two equivalent $q$-quadratic summations for a terminating balanced ${}_4\phi_3$, namely
\begin{eqnarray}
&&\hspace{-1.35cm}\qhyp43{q^{-n},q^{n+1}a,q\sqrt{c},-\sqrt{c}}{q^2c,q\sqrt{a},-\sqrt{a}}{q,q}=
\frac{(-1)^n\cdot q^{\lfloor\frac{n+1}2\rfloor-\lceil\frac{n+1}2\rceil} c^{\lfloor\frac{n+1}2\rfloor}}
{(1+\sqrt{a})(\sqrt{c}+\sqrt{a})(1+q\sqrt{c})}
\frac{(q,\frac ac;q^2)_{\lfloor\frac{n+1}2\rfloor}}{(q^2a,q^3c;q^2)_{\lfloor\frac{n}2\rfloor}}\notag\\
&&
\times
\left\{ \begin{array}{ll}
\displaystyle 
(q\sqrt{c}+\sqrt{a})(1+q^{n+1}\sqrt{a})(1+q^{n+1}\sqrt{c})-\sqrt{a}(1-q^{n+1})(1-q^{n+2}c)
& \ \mathrm{if}\ n\ \mathrm{even},\\[10pt]
\displaystyle 1 & \ \mathrm{if}\ n\ \mathrm{odd},
\end{array} \right.
\label{eq:f1}
\end{eqnarray}
\begin{eqnarray}
&&\qhyp43{q^{-n},q^{n+1}a,q\sqrt{c},-q^2\sqrt{c}}{q^2c,q\sqrt{a},-q^2\sqrt{a}}{q,q}=
\frac{q^{2\lfloor\frac{n+1}2\rfloor-1} c^{\lfloor\frac{n+1}2\rfloor}\,(1+q
\sqrt{a})}{(1+q^n\sqrt{a})(1+q^{n+1}\sqrt{a})(\sqrt{c}+\sqrt{a})(1+q\sqrt{c})}
\frac{(q,\frac ac;q^2)_{\lfloor\frac{n+1}2\rfloor}}{(q^2a,q^3c;q^2)_{\lfloor\frac{n}2\rfloor}}\notag\\
&&\hspace{1.5cm}
\times
\left\{ \begin{array}{ll}
\displaystyle 
(q\sqrt{c}+\sqrt{a})(1+q^{n+1}\sqrt{a})(1+q^{n+1}\sqrt{c})-\sqrt{a}(1-q^{n+1})(1-q^{n+2}c)
& \ \mathrm{if}\ n\ \mathrm{even},\\[10pt]
\displaystyle 1 & \ \mathrm{if}\ n\ \mathrm{odd},
\end{array} \right.
\label{eq:f1r}
\end{eqnarray}
where the two summations in \eqref{eq:f1} and \eqref{eq:f1r} can be shown
to be equivalent using 
Sears' balanced terminating ${}_4\phi_3$ transformation 
\cite[\href{http://dlmf.nist.gov/17.9.E14}{(17.9.14)}]{NIST:DLMF}.
\end{thm}

\begin{proof}
It suffices to prove \eqref{eq:f1}. 
Just as in the proofs of Theorems~\ref{thm34new} and
\ref{thm35new}, we use the
contiguous relation \cite[(3.3)]{Krattenthaler1993}.
This yields
\begin{multline*}
\qhyp43{q^{-n},q^{n+1}a,q\sqrt{c},-\sqrt{c}}{q^2c,q\sqrt{a},-\sqrt{a}}{q,q}=
\qhyp43{q^{-n},q^{n+1}a,\ppm\sqrt{c}}{q^2c,\ppm\sqrt{a}}{q,q}\\
+\frac{q(\sqrt{c}-\sqrt{a})(1-q^{-n})(1-q^{n+1}a)(1+\sqrt{c})}{(1-q^2c)(1-q\sqrt{a})(1-a)}
\qhyp43{q^{1-n},q^{n+2}a,\ppm q\sqrt{c}}{q^3c,q^2\sqrt{a},-q\sqrt{a}}{q,q}.
\end{multline*}
Now the two $_4\phi_3$ series on the right-hand side can be
evaluated by \eqref{eq:n5} and \eqref{eq:nn2}, respectively.
The two terms can be combined to the right-hand side in \eqref{eq:f1}.
\end{proof}

\begin{thm}
\label{thmf2}
Let $q,a,c\in\CCast$. Then one has the following two equivalent $q$-quadratic summations for a terminating balanced ${}_4\phi_3$, namely
\begin{eqnarray}
&&\hspace{-.5cm}\qhyp43{q^{-n},q^{n+1}a,q\sqrt{c},-\sqrt{c}}{c,q\sqrt{a},-q^2\sqrt{a}}{q,q}=
\frac{(-1)\cdot c^{\lfloor\frac{n+1}2\rfloor}\,(1+q\sqrt{a})}{(\sqrt{c}-q\sqrt{a})(1-\sqrt{c})(1+q^n\sqrt{a})(1+q^{n+1}\sqrt{a})}
\frac{(q,\frac{q^2a}c;q^2)_{\lfloor\frac{n+1}2\rfloor}}{(q^2a,qc;q^2)_{\lfloor\frac{n}2\rfloor}}\notag\\
&&\hspace{1.5cm}
\times
\left\{ \begin{array}{ll}
\displaystyle 
(\sqrt{a}-\sqrt{c})(1+q^{n+1}\sqrt{a})(1-q^{n}\sqrt{c})-
\sqrt{a}(1-q^{n+1})(1-q^{n}c)
& \ \mathrm{if}\ n\ \mathrm{even},\\[10pt]
\displaystyle 1 & \ \mathrm{if}\ n\ \mathrm{odd},
\end{array} \right.
\label{eq:f2}
\\
&&\hspace{-.5cm}\qhyp43{q^{-n},q^{n+1}a,q^{-1}\sqrt{c},-\sqrt{c}}{c,q\sqrt{a},-\sqrt{a}}{q,q}=
\frac{(-1)^{n+1}q^{-n}c^{\lfloor\frac{n+1}2\rfloor}}{(1+\sqrt{a})(\sqrt{c}-q\sqrt{a})(1-\sqrt{c})}
\frac{(q,\frac{q^2a}c;q^2)_{\lfloor\frac{n+1}2\rfloor}}{(q^2a,qc;q^2)_{\lfloor\frac{n}2\rfloor}}\notag\\
&&\hspace{1.5cm}
\times
\left\{ \begin{array}{ll}
\displaystyle 
(\sqrt{a}-\sqrt{c})(1+q^{n+1}\sqrt{a})(1-q^{n}\sqrt{c})-
\sqrt{a}(1-q^{n+1})(1-q^{n}c)
& \ \mathrm{if}\ n\ \mathrm{even},\\[10pt]
\displaystyle 1 & \ \mathrm{if}\ n\ \mathrm{odd},
\end{array} \right.
\label{eq:f2r}
\end{eqnarray}
where the two summations in \eqref{eq:f2} and \eqref{eq:f2r} can be shown
to be equivalent using 
Sears' balanced terminating ${}_4\phi_3$ transformation 
\cite[\href{http://dlmf.nist.gov/17.9.E14}{(17.9.14)}]{NIST:DLMF}.
\end{thm}
\begin{proof}
It suffices to prove \eqref{eq:f2r}. 
Again we use the
contiguous relation \cite[(3.3)]{Krattenthaler1993}
(but now in the other direction!).
This yields
\begin{multline*}
\qhyp43{q^{-n},q^{n+1}a,q^{-1}\sqrt{c},-\sqrt{c}}{c,q\sqrt{a},-\sqrt{a}}{q,q}=
\qhyp43{q^{-n},q^{n+1}a,\ppm\sqrt{c}}{qc,q\sqrt{a},-\sqrt{a}}{q,q}\\
-\frac{\sqrt{c}(1-q\sqrt{c})(1-q^{-n})(1-q^{n+1}a)}
{(1-\sqrt{c})(1-qc)(1-q\sqrt{a})(1+\sqrt{a})}
\qhyp43{q^{1-n},q^{n+2}a,\sqrt{c},-q\sqrt{c}}{q^2c,q^2\sqrt{a},-q\sqrt{a}}{q,q}.
\end{multline*}
Now the two $_4\phi_3$ series on the right-hand side can be
evaluated by \eqref{eq:nn2} and, inductively, by an instance of
\eqref{eq:f2r}, respectively.
The two terms can be combined to the right-hand side in \eqref{eq:f2r}.
\end{proof}

\begin{rem}
Since all summations given in this section are given in terms of balanced ${}_4\phi_3$'s they have corresponding representations in terms of Askey--Wilson polynomials (see Theorem \ref{AWthm}). When these summations are written in terms of Askey--Wilson polynomials one may be able to ascertain that certain of these summations are equivalent. In particular, Theorem \ref{thmf1} and Theorem \ref{thmf2}, can be seen to be equivalent summations when viewed in terms of their representations in terms of Askey--Wilson polynomials. For instance if one uses the first part of Theorem \ref{thmf1}, makes the replacement $(a,c)\mapsto(a^2b^2,q^{-1}b^2)$ followed by replacing $(a,b)\mapsto(-a,-b)$ then one finds
that the parameters for the corresponding Askey--Wilson polynomials are $(w,a,b,c,d)=(iq^{\frac12},ib,-iqb,-iqa,ia)$. For instance, if one uses the first part of Theorem \ref{thmf2}, makes the replacement $(a,c)\mapsto(a^2b^2,qb^2)$ then one finds that the parameters of the corresponding Askey--Wilson polynomials are $(w,a,b,c,d)=(iq^{\frac12},-iqb,ib,ia,-iqa)$. Because the Askey--Wilson polynomials are invariant under $a,b,c,d$ parameter interchange, we see that both of these summation theorems lead to the same summation formula for Askey--Wilson polynomials, namely Theorem \ref{thmAW38} below. So in this sense, both Theorem \ref{thmf1} and Theorem \ref{thmf2} can be considered to be equivalent.
In the same sense, one can see that Theorem \ref{thm35new} is also equivalent to Theorem \ref{thm36new}.
\end{rem}

\subsection{New terminating 2-balanced and 3-balanced ${}_4\phi_3$ summations}\label{subsec:23bal}

The following terminating 2-balanced ${}_4\phi_3$ summation is not new but due to Wei, Gong and Li \cite[Corollary~2]{WGL2013},
nevertheless we display it in this section, together with a direct proof.
\begin{thm}
\label{thm:n3}
Let $q,a,c\in\CCast$. Then one has the following $q$-quadratic summation for a terminating 2-balanced ${}_4\phi_3$, namely 
\begin{equation}
\qhyp43{q^{-n},q^na,\ppm \sqrt{c}}{\ppm \sqrt{qa},qc}{q,q}=c^{\lfloor\frac{n+1}2\rfloor}
\frac{(q;q^2)_{\lfloor\frac{n+1}2\rfloor}(\frac{qa}{c};q^2)_{\lfloor\frac n2\rfloor}}{(qa;q^2)_{\lfloor\frac n2\rfloor}(qc;q^2)_{\lfloor\frac{n+1}2\rfloor}}.
\label{eq:n3}
\end{equation}
\end{thm}
\begin{proof}
This result can be proved by induction, using the simple decomposition
\begin{equation}\label{eq:elid2}
\frac {(1-c)}{(1-q^kc)}=1-c\frac{(1-q^k)}{(1-q^kc)}
\end{equation}
and the balanced $_4\phi_3$ summation in \eqref{Andrews2}. The details are routine and are omitted.
\end{proof}
\noindent We now provide two new 2-balanced ${}_4\phi_3$ summations.
\begin{thm}
\label{thm:n4}
Let $q,a,c\in\CCast$. Then one has the following $q$-quadratic summation for a terminating 2-balanced ${}_4\phi_3$, namely 
\begin{equation}
\qhyp43{q^{-n},q^na,\ppm \sqrt{c}}{\ppm q\sqrt{a},c}{q,q}=(q^na)^n\Big(\frac{q^{-2n}c}{a^2}\Big)^{\lfloor\frac n2\rfloor}
\frac{(1-a)}{(1-aq^{2n})}
\frac{(q;q^2)_{\lfloor\frac{n+1}2\rfloor}(\frac{q^2a}{c};q^2)_{\lfloor\frac n2\rfloor}}{(a;q^2)_{\lfloor\frac{n+1}2\rfloor}(qc;q^2)_{\lfloor\frac n2\rfloor}}.
\label{eq:n4}
\end{equation}
\end{thm}
\begin{proof}
The proof is similar to that of Theorem~\ref{thm:n3}.
One proceeds by induction and uses the $c\mapsto c/q$ case of \eqref{eq:elid2},
together with the inductive hypothesis and the $c\mapsto c/q^2$ instance of \eqref{eq:n1}.
\end{proof}
\begin{thm}
\label{thm:310}
Let $q,a,c\in\CCast$. Then one has the following $q$-quadratic summation for a terminating 2-balanced ${}_4\phi_3$, namely 
\begin{eqnarray}
&&\hspace{-2cm}\qhyp43{q^{-n},q^{n-1}a,\ppm q\sqrt{c}}{\ppm q\sqrt{a},qc}{q,q}=
\frac{(-1)^n\,q^n\,c^{\lfloor\frac n2\rfloor}\,(1-a)}{(1-q^{2n}a)(1-q^{2n-2}a)}
\frac{(q;q^2)_{\lfloor\frac{n+1}2\rfloor}(\frac ac;q^2)_{\lfloor\frac n2\rfloor}}{(a;q^2)_{\lfloor\frac n2\rfloor}(qc;q^2)_{\lfloor\frac{n+1}2\rfloor}}\notag\\
&&\hspace{3.5cm}\times\left\{ \begin{array}{ll}
\displaystyle c(1+q^{2n-1}a)-q^{n-2}a(1+q) & \ \mathrm{if}\ n\ \mathrm{odd},\\[5pt]
\displaystyle 
(1+q^{2n-1}a)-q^{n-2}a(1+q)
& \ \mathrm{if}\ n\ \mathrm{even}.
\end{array} \right.
\label{eq:n8}
\end{eqnarray}
\end{thm}

\begin{proof}
The proof uses induction, the $c\mapsto q^{n-1}a$ case of \eqref{eq:elid2},
the inductive hypothesis and \eqref{eq:n2}.
 \end{proof}

\noindent Similar to \eqref{eq:n5}, the summation in \eqref{eq:n8} can be viewed as esoteric.
\medskip

\noindent Further, we have derived
two new 3-balanced terminating ${}_4\phi_3$ summations. The following summation is somewhat esoteric but possesses a complete product evaluation for odd $n$.

\begin{thm}
\label{thm:n7}
Let $q,a,c\in\CCast$. Then one has the following $q$-quadratic summation for a terminating 3-balanced ${}_4\phi_3$, namely
\begin{eqnarray}
&&\hspace{0cm}\qhyp43{q^{-n},q^{n-1}a,\ppm \sqrt{c}}{\ppm q\sqrt{a},c}{q,q}=
\frac{c^{\lfloor\frac{n+1}2\rfloor}}{(c-a)(1-q^{2n}a)(1-q^{2n-2}a)}
\frac{(q;q^2)_{\lfloor\frac{n+1}2\rfloor}(\frac ac;q^2)_{\lfloor\frac n2\rfloor}}{(q^2a;q^2)_{\lfloor\frac{n+1}2\rfloor}(qc;q^2)_{\lfloor\frac n2\rfloor}}\notag\\
&&\hspace{1.0cm}\times\left\{ \begin{array}{ll}
\displaystyle 
(1-q^na)\big(q^{2n-2}a(a-c)(1-q^{2n}a)+(c-q^na)(1+q^{2n-1}a)(1-q^{n-1}a)\big)
& \ \mathrm{if}\ n\ \mathrm{even},\\[5pt]
\displaystyle q^{n-1}a(1+q)(1-q^{n+1}a)(1-q^{n-1}a)(1-q^{n-1}\tfrac ac) & \ \mathrm{if}\ n\ \mathrm{odd}.
\end{array} \right.
\label{eq:n7}
\end{eqnarray}
\end{thm}

\begin{proof}
The proof uses induction, the $c\mapsto q^{n-1}a$ case of \eqref{eq:elid2},
the inductive hypothesis and Theorem~\ref{thm:n4}.
\end{proof}

\noindent Some further $k$-balanced $_4\phi_3$ summations, closely related (more specifically, contiguous)
to the ones in the subsection, are given by Wei, Gong and Li in \cite[Corollaries 3 and 6]{WGL2013}. The $_4\phi_3$ summations in this subsection that do not completely factorize (typically
depending on the parity of $n$)
can be specialized such that the right-hand sides do factorize into closed form. There may be different ways
to achieve this. We conclude this subsection with two such examples.
The following summation for a terminating 2-balanced ${}_4\phi_3$ series arises from Theorem~\ref{thm:310}
directly by the specialization $a\mapsto -q^{1-2n}$.
It is easy to see that in this case
the right-hand side completely factorizes into closed form without division into cases.
\begin{cor}
\label{cor:310s}
Let $q,c\in\CCast$. Then one has the following $q$-quadratic summation for a terminating 2-balanced ${}_4\phi_3$, namely
\begin{equation}
\qhyp43{\ppm q^{-n},\ppm q\sqrt{c}}{\ppm {\mathrm i}\,q^{\frac32-n},qc}{q,q}=
(-1)^n\frac{(1+q^{1-2n})}{(1+q)}
\frac{(q;q^2)_{\lfloor\frac{n+1}2\rfloor}(-q^{1+2\lfloor\frac{n+1}2\rfloor}c;q^2)_{\lfloor\frac n2\rfloor}}{(qc;q^2)_{\lfloor\frac{n+1}2\rfloor}(-q^{1+2\lfloor\frac{n+1}2\rfloor};q^2)_{\lfloor\frac n2\rfloor}}.
\label{eq:310s}
\end{equation}
\end{cor}

The following summation for a terminating 3-balanced ${}_4\phi_3$ series arises from Theorem~\ref{thm:n7}
directly by the specialization $a\mapsto -q^{1-2n}$.
Also in this case, the right-hand side completely factorizes into closed form without division into cases.
\begin{cor}
\label{cor:n6}
Let $q,c\in\CCast$. Then one has the following $q$-quadratic summation for a terminating 3-balanced ${}_4\phi_3$, namely
\begin{equation}
\qhyp43{\ppm q^{-n},\ppm \sqrt{c}}{\ppm {\mathrm i}\,q^{\frac32-n},c}{q,q}=
(-1)^n\frac{(1+q^{1-2n})}{(1+q)}
\frac{(q;q^2)_{\lfloor\frac{n+1}2\rfloor}(-q^{1+2\lfloor\frac n2\rfloor}c;q^2)_{\lfloor\frac n2\rfloor}}{(qc;q^2)_{\lfloor\frac n2\rfloor}(-q^{1+2\lfloor\frac{n+1}2\rfloor};q^2)_{\lfloor\frac n2\rfloor}}.
\label{eq:n6}
\end{equation}
\end{cor}

\begin{rem}
We believe that to find direct proofs of Corollaries~\ref{cor:310s} and \ref{cor:n6},
say by induction, would be quite a challenge, even though
terminating summations such as the summations in \eqref{eq:310s} and \eqref{eq:n6}
can principally be proved in an automatic fashion
using computer algebra with implementations of well-known algorithms (and their $q$-analogues),
such as those due to Gosper, Zeilberger, and Petkov{\v{s}}ek (cf.\ \cite{PauleSchorn,PetWilfZeil}).
\end{rem}

\subsection{$q$-Quadratic special values for Askey--Wilson polynomials}

\noindent We now derive ten $q$-quadratic special values for Askey--Wilson polynomials $p_n(x;a,b,c,d|q)$:~(1) 5 with argument $x=0$; (2) 3 with argument $x=\frac{\mathrm i}{2}(q^\frac12\!-\!q^{-\frac12})$; 
(3) 1 with argument $x=\frac{1}{2}(q^\frac14\!+\!q^{-\frac14})$; and (4) 1 with argument $x=\frac12(qa\!+\!(qa)^{-1})$.

\begin{rem}Note that we arrange the following special values such that closed-form representations are given separately for $n$ even with $q$-shifted factorial index and powers given as $\tfrac12 n$; and $n$ odd with $q$-shifted factorial index and powers given as $\tfrac12(n-1)$. This way, when these special values are applied to linear and bilinear generating functions such as the Ismail--Wilson generating function, after one splits the series into contributions over $n$ even and $n$ odd, one is already in a form suitable to derive nonterminating basic hypergeometric representations   which correspond to parity values.
\end{rem}

\medskip

\noindent Corresponding to 
Bailey's summation \eqref{newsum3} (and equivalently \eqref{Andrews2}, \eqref{eq:Jain}), one has the following $q$-quadratic special value for Askey--Wilson polynomials.
\medskip
\begin{thm}
Let $n\in\N_0$, $q,a,b\in\CCast$. Then
\begin{eqnarray}
&&\hspace{-4.7cm}p_n(0;\ppm{\mathrm i}a,\ppm{\mathrm i}b|q)=
\left\{ \begin{array}{ll}
\displaystyle (-1)^{\frac{n}{2}}\frac{(q,a^2,b^2,\ppm ab,\ppm qab;q^2)_{\frac{n}{2}}}{(a^2b^2;q^2)_{\frac{n}{2}}} & \ \mathrm{if}\ n\ \mathrm{even},\\[10pt]
\displaystyle 
0
& \ \mathrm{if}\ n\ \mathrm{odd}.
\end{array} \right.
\label{Bailey}
\end{eqnarray}
\end{thm}
\begin{proof}
Starting with \eqref{newsum3} then comparing with \eqref{aw:def2} with $(w,a,b,c,d)\mapsto(-{\mathrm i},{\mathrm i}a,{\mathrm i}b,-{\mathrm i}a,-{\mathrm i}b)$, then replacing ${\mathrm i}\mapsto-{\mathrm i}$. Let $m\in\N_0$. Finally
replacing $n=2m$ for $n$ even, and $n=2m+1$ for $n$ odd, followed by applying standard identities for $q$-shifted factorials, and solving for $m$ in terms of $n$ completes the proof.
\end{proof}
\noindent Corresponding to our new balanced ${}_4\phi_3$ summation given in Theorem \ref{thm31} we derive the following $q$-quadratic special value for Askey--Wilson polynomials.
\begin{thm}
\label{thm318}
Let $n\in\N_0$, $q,a,b\in\CCast$. Then the Askey--Wilson polynomials have the following $q$-quadratic special value,
\begin{eqnarray}
&&\hspace{-0.7cm}p_n(0;\ppm{\mathrm i}a,{\mathrm i}b,-{\mathrm i}qb|q)\nonumber\\
&&\hspace{0.0cm}=(-{\mathrm i})^nb^{2\lfloor\frac{n+1}{2}\rfloor-n}
\frac{(qb^2,\ppm ab;q)_n(q,a^2;q^2)_{\lfloor\frac{n+1}{2}\rfloor}}{(qb^2,a^2b^2;q^2)_{\lfloor\frac{n+1}{2}\rfloor}}\label{firstit}\\
&&=
\left\{ 
\begin{array}{ll}
\displaystyle 
(-1)^{\frac{n}{2}}\frac{(q,a^2,q^2b^2,\ppm ab,\ppm qab;q^2)_{\frac{n}{2}}}{(a^2b^2;q^2)_{\frac{n}{2}}} &
\qquad\mathrm{if}\ n\ \mathrm{even},
\\[5pt]
\displaystyle 
-{\mathrm i}(-1)^{\frac{n-1}{2}}(1-q)(1-a^2)b
\frac{(q^3,q^2a^2,q^2b^2,\ppm qab,\ppm q^2ab;q^2)_{\frac{n-1}{2}}}{(q^2a^2b^2;q^2)_{\frac{n-1}{2}}}&
\qquad\mathrm{if}\ n\ \mathrm{odd}.
\end{array} 
\right.
\label{eqthm318}
\end{eqnarray}
\end{thm}
\begin{proof}
Replace 
\[
(a,c)\mapsto(a^2b^2,b^2)
\]
in either \eqref{eq:n2}, \eqref{eq:n1} then setting
\[
(w,a,b,c,d)\mapsto ({\mathrm i},{\mathrm i}a,-{\mathrm i}a,{\mathrm i}b,-{\mathrm i}qb)
\]
in \eqref{aw:def1} produces \eqref{firstit}.
Then replacing $n=2m$ for $n$ even, and $n=2m+1$ for $n$ odd, followed by standard identities for $q$-shifted factorials, finally solving for $m$ in terms of $n$ completes the proof.
\end{proof}
\noindent Corresponding to our new balanced ${}_4\phi_3$ summation given in Theorem \ref{thm34new} we derive the following $q$-quadratic special value for Askey--Wilson polynomials.
\begin{thm}
\label{thm319}
Let $n\in\N_0$, $q,a,b\in\CCast$. Then the Askey--Wilson polynomials have the following $q$-quadratic special value,
\begin{eqnarray}
&&\hspace{-0.2cm}p_n(0;{\mathrm i}a,-{\mathrm i}qa,{\mathrm i}b,-{\mathrm i}qb|q)\nonumber\\
&&=
\left\{ 
\begin{array}{ll}
\displaystyle 
(-1)^{\frac{n}{2}}\frac{(q,q^2a^2,q^2b^2,qab,-ab,q^2ab,-qab;q^2)_{\frac{n}{2}}}{(q^2a^2b^2;q^2)_{\frac{n}{2}}} &
\qquad\mathrm{if}\ n\ \mathrm{even},
\\[5pt]
\displaystyle 
-{\mathrm i}(-1)^{\frac{n-1}{2}}(1-q)(a+b)(1-qab)
\frac{(q^3,q^2a^2,q^2b^2,q^2ab,-qab,q^3ab,-q^2ab;q^2)_{\frac{n-1}{2}}}{(q^2a^2b^2;q^2)_{\frac{n-1}{2}}}&
\qquad\mathrm{if}\ n\ \mathrm{odd}.
\end{array} 
\right.
\label{eqthm319}
\end{eqnarray}
\end{thm}
\begin{proof}
Replace 
\[
(a,c)\mapsto(a^2b^2,a^2)
\]
then replace $(a,b)\mapsto(-a,-b)$ 
in either \eqref{eq:nn2}, \eqref{eq:nn2r} then setting
\[
(w,a,b,c,d)\mapsto ({\mathrm i},{\mathrm i}qa,-{\mathrm i}a,{\mathrm i}qb,-{\mathrm i}b)
\]
in \eqref{aw:def1} and replacing $n=2m$ for $n$ even, and $n=2m+1$ for $n$ odd, followed by standard identities for $q$-shifted factorials, finally solving for $m$ in terms of $n$ completes the proof.
\end{proof}
\noindent Corresponding to Andrews' $q$-analogue of the terminating version of Whipple's ${}_3F_2(1)$ summation \eqref{Andrews43} (cf.\
\cite[\href{http://dlmf.nist.gov/16.4.E7}{(16.4.7)}]{NIST:DLMF}), one has the following $q$-quadratic special value for Askey--Wilson polynomials.
\begin{thm}
\label{thm320}
Let $n\in\N_0$, $q,a,b\in\CCast$. Then
\begin{eqnarray}
&&\hspace{-0.80cm}p_n(0;{\mathrm i}a,\tfrac{{\mathrm i}q}{a},{\mathrm i}b,\tfrac{{\mathrm i}q}{b}|q)\nonumber\\
&&\hspace{0.0cm}=\left\{ \begin{array}{ll}
\displaystyle (-1)^{\frac{n}{2}}(-q,-q^2,-ab,-\tfrac{q^2}{ab},-\tfrac{qa}{b},-\tfrac{qb}{a};q^2)_{\frac{n}{2}} & \ \mathrm{if}\ n\ \mathrm{even},\\[10pt]
\displaystyle 
-\tfrac{{\mathrm i}q}{b}(-1)^{\frac{n-1}{2}}(1+q)(1+\tfrac{ab}{q})(1+\tfrac{b}{a})(-q^2,-q^3,-qab,-\tfrac{q^3}{ab},-\tfrac{q^2a}{b},-\tfrac{q^2b}{a};q^2)_{\frac{n-1}{2}}
& \ \mathrm{if}\ n\ \mathrm{odd}.
\end{array} \right.
\label{ev:Andrews}
\end{eqnarray}
\end{thm}
\begin{proof}
Starting with \eqref{Andrews43} then comparing with \eqref{aw:def1} with $(w,a,b,c,d)\mapsto(-{\mathrm i},{\mathrm i}a,{\mathrm i}q/a,-{\mathrm i}b/a,-{\mathrm i}qa/b)$ followed by $i\mapsto -{\mathrm i}$ and $b\mapsto ab$ and $b\mapsto -b$. Let $m\in\N_0$. Replacing $n=2m$ for $n$ even, and $n=2m+1$ for $n$ odd, followed by applying standard identities for $q$-shifted factorials, and solving for $m$ in terms of $n$ completes the proof.
\end{proof}

\noindent Corresponding to our new balanced esoteric ${}_4\phi_3$ summation given in Theorem \ref{thm33} we derive the following $q$-quadratic special value for Askey--Wilson polynomials.

\begin{thm}
\label{thm44}
Let $n\in\N_0$, $q,a,b\in\CCast$. Then the Askey--Wilson polynomials have the following esoteric $q$-quadratic special value,
\begin{eqnarray}
&&\hspace{-0.5cm}p_n(0;\ppm{\mathrm i}a,{\mathrm i}b,-{\mathrm i}q^2b|q)\nonumber\\
&&\hspace{-0.2cm}=
\left\{ 
\begin{array}{ll}
\displaystyle 
\frac{(-1)^{\frac{n}{2}}(a^2,q^2b^2,\ppm ab,\ppm qab;q^2)_{\frac{n}{2}}}{(1-q^2b^2)(1-a^2b^2)(q^2a^2b^2;q^2)_{\frac{n}{2}}}&
\\[9pt] 
\displaystyle \hspace{0.25cm}\times\left(\frac{(1-qb^2)(1-qa^2b^2)(q,q^3b^2,q^3a^2b^2;q)_{\frac{n}{2}}}{(qb^2,qa^2b^2;q^2)_{\frac{n}{2}}}+\frac{qb^2(1-q)(1-\frac{a^2}{q})(q^3,qa^2;q^2)_{\frac{n}{2}}}{(\frac{a^2}{q};q^2)_{\frac{n}{2}}}\right)
&
\!\!\! \mathrm{if}\ n\ \mathrm{even},
\\[15pt]
\displaystyle -{\mathrm i}(-1)^{\frac{n-1}{2}}b(1-q^2)(1-a^2)
\frac{(q^3,q^2a^2,q^4b^2,\ppm qab,\ppm q^2ab;q^2)_{\frac{n-1}{2}}}{(q^2a^2b^2;q^2)_{\frac{n-1}{2}}}&
\!\!\! \mathrm{if}\ n\ \mathrm{odd}.
\end{array} 
\right.
\label{esoterAW}
\end{eqnarray}
\end{thm}
\begin{proof}
Replacing $(a,c)\mapsto(a^2b^2,a^2)$ in \eqref{eq:n5} and then comparing with \eqref{aw:def1} using $(w,a,b,c,d)=({\mathrm i},\ppm{\mathrm i}a,{\mathrm i}b,-{\mathrm i}q^2b)$.
Replacing $n=2m$ for $n$ even, and $n=2m+1$ for $n$ odd, followed by applying standard identities for $q$-shifted factorials, and solving for $m$ in terms of $n$ completes the proof.
\end{proof}

\noindent Corresponding to our new balanced ${}_4\phi_3$ summation given in Theorem \ref{thm33new} we derive the following $q$-quadratic special value for Askey--Wilson polynomials with argument on the imaginary axis.
\begin{thm}
\label{thm320new}
Let $n\in\N_0$, $q,a,b\in\CCast$. Then the Askey--Wilson polynomials have the following $q$-quadratic special value,
\begin{eqnarray}
&&\hspace{-0.8cm}p_n(\tfrac{\mathrm i}{2}(q^\frac12-q^{-\frac12});\ppm{\mathrm i}a,\ppm{\mathrm i}b|q)\nonumber\\
&&\hspace{0.5cm}=q^{-\frac{n}{2}}
\left\{ 
\begin{array}{ll}
\displaystyle 
(-1)^{\frac{n}{2}}\frac{(q,qa^2,qb^2,\ppm ab,\ppm qab;q^2)_{\frac{n}{2}}}{(a^2b^2;q^2)_{\frac{n}{2}}} &
\qquad\mathrm{if}\ n\ \mathrm{even},
\\[5pt]
\displaystyle 
-{\mathrm i}(-1)^{\frac{n-1}{2}}(1-q)(1-a^2b^2)
\frac{(q^3,qa^2,qb^2,\ppm qab,\ppm q^2ab;q^2)_{\frac{n-1}{2}}}{(a^2b^2;q^2)_{\frac{n-1}{2}}}&
\qquad\mathrm{if}\ n\ \mathrm{odd}.
\end{array} 
\right.
\label{eqthm320new}
\end{eqnarray}
\end{thm}
\begin{proof}
Replace 
\[
(a,c)\mapsto(a^2b^2/q^2,a^2/q)
\]
in \eqref{eq:nn3} then setting
\[
(w,a,b,c,d)\mapsto ({\mathrm i}q^{\frac12},\ppm{\mathrm i}a,\ppm{\mathrm i}b)
\]
in \eqref{aw:def1} and replacing $n=2m$ for $n$ even, and $n=2m+1$ for $n$ odd, followed by standard identities for $q$-shifted factorials, finally solving for $m$ in terms of $n$ completes the proof.
\end{proof}

\noindent 
Corresponding to our new balanced ${}_4\phi_3$ summation given in Theorem \ref{thm35new} we derive the following $q$-quadratic special value for Askey--Wilson polynomials with argument on the imaginary axis.
\begin{thm}
\label{thm321new}
Let $n\in\N_0$, $q,a,b\in\CCast$. Then the Askey--Wilson polynomials have the following $q$-quadratic special value,
\begin{eqnarray}
&&\hspace{0.05cm}p_n(\tfrac{\mathrm i}{2}(q^\frac12\!-\!q^{-\frac12});\ppm{\mathrm i}a,{\mathrm i}b,-{\mathrm i}qb|q)=q^{-\frac{n}{2}}\nonumber\\
&&\hspace{0.5cm}\times
\left\{ 
\begin{array}{ll}
\displaystyle 
(-1)^{\frac{n}{2}}\frac{(q,qa^2,q^{\frac52}b,qb^2,\ppm ab,\ppm qab;q^2)_{\frac{n}{2}}}{(q^{\frac12}b,a^2b^2;q^2)_{\frac{n}{2}}} &
\quad\mathrm{if}\ n\ \mathrm{even},
\\[5pt]
\displaystyle 
-{\mathrm i}(-1)^{\frac{n-1}{2}}(1-q)(1+q^{\frac12}b)(1-q^{\frac12}a^2b)
\frac{(q^3,qa^2,q^3b^2,\ppm qab,\ppm q^2ab,q^{\frac52}a^2b;q^2)_{\frac{n-1}{2}}}{(q^\frac12a^2b,q^2a^2b^2;q^2)_{\frac{n-1}{2}}}&
\quad\mathrm{if}\ n\ \mathrm{odd}.
\end{array} 
\right.
\label{eqthm321new}
\end{eqnarray}
\end{thm}
\begin{proof}
Replace 
\[
(a,c)\mapsto(a^2b^2,a^2/q)
\]
followed by $b\mapsto-b$ 
in \eqref{eq:nn1} then setting
\[
(w,a,b,c,d)\mapsto ({\mathrm i}q^{\frac12},\ppm{\mathrm i}a,{\mathrm i}b,-{\mathrm i}qb)
\]
in \eqref{aw:def1} and replacing $n=2m$ for $n$ even, and $n=2m+1$ for $n$ odd, followed by standard identities for $q$-shifted factorials, finally solving for $m$ in terms of $n$ completes the proof.
\end{proof}

\noindent One has the following $q$-quadratic special value for Askey--Wilson polynomials which corresponds to Theorem \ref{thmf1}.
\begin{thm}
\label{thmAW38}
Let $n\in\N_0$, $q,a,b\in\CCast$. Then the Askey--Wilson polynomials have the following special value,
\begin{eqnarray}
&&\hspace{0.1cm}p_n(\tfrac{\mathrm i}2(q^\frac12\!-\!q^{-\frac12});{\mathrm i}a,-{\mathrm i}qa,{\mathrm i}b,-{\mathrm i}qb|q)\nonumber\\
&&\hspace{0.5cm}=
\left\{ \!\!
\begin{array}{ll}
\displaystyle 
\frac{(-1)^{\frac{n}{2}}(q,qa^2,qb^2,-ab,\ppm qab,q^2ab;q^2)_{\frac{n}{2}}}{q^{\frac{n}{2}+1}(1+ab)(1-\frac{1}{q^{\frac12}a})(q^2a^2b^2;q^2)_{\frac{n}{2}}}\Biggl(\frac{(1-\frac{q^{\frac12}}{a})(1+qab)(q^\frac52b,-q^3ab;q^2)_{\frac{n}{2}}}{(q^\frac12b,-qab;q^2)_{\frac{n}{2}}}&\\
\hspace{8cm}\displaystyle -\frac{(1-q)(1+q^\frac12b)(q^3,q^3b^2;q^2)_{\frac{n}{2}}}{(q,qb^2;q^2)_{\frac{n}{2}}}\Biggr) 
&\quad\mathrm{if}\ n\ \mathrm{even},
\\[30pt]
\displaystyle -{\mathrm i}q^{-\frac{n}{2}}(-1)^{\frac{n-1}{2}}(1-q)(1+q^\frac12a)(1+q^\frac12b)(1-qab)\\[12pt]
\displaystyle \hspace{6cm}\times \frac{(q^3,q^3a^2,q^3b^2,-qab,\ppm q^2ab,q^3ab;q^2)_{\frac{n-1}{2}}}{(q^2a^2b^2;q^2)_{\frac{n-1}{2}}}&
\quad\mathrm{if}\ n\ \mathrm{odd}.
\end{array} 
\right.
\label{eqthmAW38}
\end{eqnarray}
\end{thm}
\begin{proof}
Start for instance with the first part of Theorem \ref{thmf1}, identify the parameters of the Askey--Wilson polynomials using \eqref{aw:def1} and scale the parameters appropriately. Splitting into even and odd parts completes the proof.
\end{proof}

\noindent One has the following special value for the Askey--Wilson polynomials which follows from the equivalent summations \eqref{GR-3.10.9}, \eqref{GR-3.10.10}.
\begin{thm}
\label{thm321xnew}
Let $n\in\N_0$, $q,a,b\in\CCast$. Then the Askey--Wilson polynomials have the following $q$-quadratic special value,
\begin{eqnarray}
&&\hspace{-0.75cm}p_n(\tfrac12(qa+\tfrac{1}{qa});\ppm q^\frac12,a,b|q)\nonumber\\
&&\hspace{0.5cm}=(qa)^{-n}
\left\{ \!\!
\begin{array}{ll}
\displaystyle 
\frac{(\ppm q^\frac12a,\ppm q^\frac32a,ab,qab,\ppm q^2\sqrt{ab};q^2)_{\frac{n}{2}}}{(\ppm\sqrt{ab};q^2)_{\frac{n}{2}}} &
\quad\mathrm{if}\ n\ \mathrm{even},
\\[5pt]
\displaystyle 
\frac{(1-qa^2)(1-q^2ab)(\ppm q^\frac32a,\ppm q^\frac52a,qab,q^2ab,\ppm q^3\sqrt{ab};q^2)_{\frac{n-1}{2}}}{(\ppm q\sqrt{ab};q^2)_{\frac{n-1}{2}}}&
\quad\mathrm{if}\ n\ \mathrm{odd}.
\end{array} 
\right.
\label{eqthm321xnew}
\end{eqnarray}
\end{thm}
\begin{proof}
Start with \eqref{GR-3.10.9} and convert to an Askey--Wilson polynomial using \eqref{aw:def3}. Breaking the right-hand side into even and odd parts and shifting the parameters accordingly completes the proof.
\end{proof}

\noindent One has the following special value for the Askey--Wilson polynomials which follows directly from the equivalent balanced  summations 
\eqref{eq:BIS-2.2}--\eqref{eq:wx} (see Remark \ref{Rem31}).
\begin{thm}
\label{thmWeiWang}
Let $n\in\N_0$, $q,a,b\in\CCast$. Then the Askey--Wilson polynomials have the following special value,
\begin{eqnarray}
&&\hspace{-0.7cm}p_n(\tfrac12(q^\frac14\!+\!q^{-\frac14});a,q^{\frac12}a,b,q^{\frac12}b|q)=q^{-\frac{n}{4}}\nonumber\\
&&\hspace{0.5cm}\times
\left\{ \!\!
\begin{array}{ll}
\displaystyle 
\frac{(-q^\frac12,-q,q^\frac14a,q^\frac34a,q^{\frac14}b,q^{\frac34}b,ab,q^{\frac12}ab,qab,q^\frac32ab;q^2)_{\frac{n}{2}}}{(ab,q^{\frac12}ab;q)_{\frac{n}{2}}} &
\quad\mathrm{if}\ n\ \mathrm{even},
\\[30pt]
\displaystyle (1+q^\frac12)(1-q^\frac14a)(1-q^{\frac14}b)(1-q^{\frac12}ab)\\[6pt]
\hspace{2cm}\displaystyle \times\frac{(-q,-q^{\frac32},q^\frac34a,q^\frac54a,q^\frac34b,q^\frac54b,q^\frac32ab,q^2ab,q^{\frac52}ab;q)_{\frac{n-1}{2}}}{(q^{\frac12}ab;q)_{\frac{n-1}{2}}}&
\quad\mathrm{if}\ n\ \mathrm{odd}.
\end{array} 
\right.
\label{eqthmWeiWang}
\end{eqnarray}
\end{thm}
\begin{proof}
Starting with for instance, the Bressoud--Ismail--Stanton summation  \eqref{eq:BIS-2.2}, let $q^2\mapsto q$ and identify the parameters of the Askey--Wilson polynomials using \eqref{aw:def1}. Splitting into even and odd parts completes the proof.
\end{proof}

\section{Nonterminating product transformations}\label{sec:npt}

\noindent 
Starting from our new $q$-quadratic special values
of the Askey--Wilson polynomials given in Theorem \ref{thm44} and Theorem \ref{thm321new}, we obtain two 4-term transformation formulas for nonterminating products of two ${}_2\phi_1$'s and two 3-term transformation formulas. 
In particular, starting from the esoteric special value
of the Askey--Wilson polynomials \eqref{esoterAW}, one
obtains the following transformations for nonterminating products of two ${}_2\phi_1$'s.

\begin{thm}
\label{thm517}
Let $0<|q|<1$, $a,c\in\CCast$, $|t|<1$. Then
\begin{eqnarray}
&&\hspace{-2.5cm}\qhyp21{-c,q^2c}{q^2c^2}{q,t}\qhyp21{\ppm a}{a^2}{q,-t}=\frac{ct(1+q)}{(1-q^2c^2)}\qhyp43{\ppm qac,\ppm q^2ac}{qa^2,q^3c^2,q^2a^2c^2}{q^2,t^2}\nonumber\\
&&\hspace{0.0cm}+\frac{(1-qc^2)(1-qa^2c^2)}{(1-q^2c^2)(1-a^2c^2)}\qhyp54{q^3a^2c^2,\ppm ac,\ppm qac}{qa^2,qc^2,qa^2c^2,q^2a^2c^2}{q^2,t^2}
\nonumber\\
&&\hspace{0.0cm}
+\frac{qc^2(1-q)(1-\frac{a^2}{q})}{(1-q^2c^2)(1-a^2c^2)}\qhyp54{q^3,\ppm ac,\ppm qac}{q,\frac{a^2}{q},q^3c^2,q^2a^2c^2}{q^2,t^2}.
\end{eqnarray}
\end{thm}
\begin{proof}
Start with the Ismail--Wilson product generating function for Askey--Wilson polynomials \eqref{AWgf}, and replace
$(w,a,b,c,d)\mapsto({\mathrm i},\ppm{\mathrm i}a,{\mathrm i}b,-{\mathrm i}q^2b)$
using Theorem \ref{thm44}. Finally splitting into even and odd terms using \eqref{esoterAW} and simplifying completes the proof.
\end{proof}

\begin{thm}
\label{thm518}
Let $0<|q|<1$, $a,b\in\CCast$, $|t|<1$. Then
\begin{eqnarray}
&&\hspace{-1.0cm}\qhyp21{-a,-b}{-ab}{q,t}\qhyp21{-a,-q^2b}{-q^2ab}{q,-t}=\frac{bt(1+q)(1-a^2)}{(1+ab)(1+q^2ab)}\qhyp43{q^2a^2,q^4b^2,qab,q^2ab}{-q^3ab,-q^4ab,q^2a^2b^2}{q^2,t^2}\nonumber\\
&&\hspace{0.0cm}+\frac{(1-qb^2)(1-qa^2b^2)}{(1-q^2b^2)(1-a^2b^2)}\qhyp65{a^2,q^2b^2,q^3b^2,ab,qab,q^3a^2b^2}{qb^2,-q^2ab,-q^3ab,qa^2b^2,q^2a^2b^2}{q^2,t^2}
\nonumber\\
&&\hspace{0.0cm}
+\frac{qb^2(1-q)(1-\frac{a^2}{q})}{(1-q^2b^2)(1-a^2b^2)}\qhyp65{q^3,a^2,qa^2,q^2b^2,ab,qab}{q,\frac{a^2}{q},-q^2ab,-q^3ab,q^2a^2b^2}{q^2,t^2}.
\end{eqnarray}
\end{thm}
\begin{proof}
Start with the Ismail--Wilson product generating function for Askey--Wilson polynomials \eqref{AWgf}, and replace
$(w,a,b,c,d)\mapsto({\mathrm i},{\mathrm i}a,{\mathrm i}b,-{\mathrm i}a,-{\mathrm i}q^2b)$
using Theorem \ref{thm44}. Finally splitting into even and odd parts using \eqref{esoterAW} and simplifying completes the proof.
\end{proof}

\noindent Starting from Theorem \ref{thm319}, we obtain the following product transformation formula.
\begin{thm}\label{thm43new}
Let $0<|q|<1$, $a,b\in\CCast$, $|t|<1$. Then
\begin{eqnarray}
&&\hspace{-0.8cm}\qhyp21{a,-qa}{qa^2}{q,t}\qhyp21{b,-qb}{qb^2}{q,-t}\nonumber\\
&&\hspace{-0.0cm}=\qhyp43{ab,\ppm qab,-q^2ab}{qa^2,qb^2,q^2a^2b^2}{q^2,t^2}+
\frac{(b-a)(1+qab)t}{(1-qa^2)(1-qb^2)}\qhyp43{qab,\ppm q^2ab,-q^3ab}{q^3a^2,q^3b^2,q^2a^2b^2}{q^2,t^2}.
\end{eqnarray}
\label{newthmfin}
\end{thm}
\begin{proof}
Start with the Ismail--Wilson product generating function for Askey--Wilson polynomials \eqref{AWgf}, and replace
$(w,a,b,c,d)\mapsto({\mathrm i},{\mathrm i}a,-{\mathrm i}qa,{\mathrm i}b,-{\mathrm i}qb)$
using Theorem \ref{thm319}. Finally splitting into even and odd terms using \eqref{eqthm318} and simplifying completes the proof.
\end{proof}

\noindent If one replaces $\{a,b,t\}\mapsto\{q^a,q^b,\tfrac12(1-q^2)z\}$ in Theorem \ref{newthmfin}, then one obtains the following transformation formula for a product of two confluent hypergeometric functions which is a transformation of type \cite[(4.3.5-6), (4.3.13)]{Erdelyi}. Note that \cite[(4.3.13)]{Erdelyi} contains a typo, the numerator element $-a'$ should be $a'$.

\begin{cor}
Let $a,b,z\in\CC$. Then
\begin{eqnarray}
&&\hspace{-0.2cm}\hyp11{a}{2a+1}{z}\hyp11{b}{2b+1}{-z}\nonumber\\
&&\hspace{0.5cm}=\hyp23{\frac{a+b}{2},\frac{a+b+1}{2}}{a+\frac12,b+\frac12,a+b+1}{\frac{z^2}{4}}+\frac{(a-b)z}{(2a+1)(2b+1)}\hyp23{\frac{a+b+1}{2},\frac{a+b+2}{2}}{a+\frac32,b+\frac32,a+b+1}{\frac{z^2}{4}}.
\end{eqnarray}
\end{cor}
\begin{proof}
Start with Theorem \ref{newthmfin} and make the replacements $\{a,b,z\}\mapsto\{q^a,q^b,\tfrac12(1-q^2)z\}$. Then take the $q\to 1^{-}$ limit and  setting $t\mapsto z/2$ completes the proof.
\end{proof}

\noindent If one replaces $b\mapsto a$ then the second term on the right-hand side vanishes and we obtain the following interesting identity
\begin{equation}
\hyp11{a}{2a+1}{t}\hyp11{a}{2a+1}{-t}=\hyp12{a}{a+\frac12,2a+1}{\frac{t^2}{4}},
\end{equation}
which is a special case of \cite[(4.3.5)]{Erdelyi}. Starting from Theorem \ref{thm321new}, we obtain the following.

\begin{thm}
\label{thm515}
Let $0<|q|<1$, $a,b\in\CCast$, $|t|<1$. Then
\begin{eqnarray}
&&\hspace{-3.5cm}\qhyp21{b,-qb}{qb^2}{q,t}\qhyp21{\ppm a}{a^2}{q,-t}\nonumber\\
&&\hspace{-2.5cm}=\qhyp43{\ppm ab,\ppm qab}{qa^2,qb^2,a^2b^2}{q^2,t^2}+\frac{-bt}{1-qb^2}\qhyp43{\ppm qab,\ppm q^2ab}{qa^2,q^3b^2,q^2a^2b^2}{q^2,t^2}.
\end{eqnarray}
\end{thm}
\begin{proof}
Start with the Ismail--Wilson product generating function for Askey--Wilson polynomials \eqref{AWgf}, and replace
$(w,a,b,c,d)\mapsto({\mathrm i},\ppm{\mathrm i}a,{\mathrm i}b,-{\mathrm i}qb)$
using Theorem \ref{thm321new}. Finally splitting into even and odd terms using \eqref{eqthm318} and simplifying completes the proof.
\end{proof}

\noindent By rearranging the coefficient of the Askey--Wilson polynomial associated with the $q$-quadratic special value used above, then one obtains the following transformation.
\begin{thm}
\label{thm516}
Let $0<|q|<1$, $a,b\in\CCast$, $|t|<1$. Then
\begin{eqnarray}
&&\hspace{-0.5cm}\qhyp21{-a,-b}{-ab}{q,t}\qhyp21{-a,-qb}{-qab}{q,-t}\nonumber\\
&&\hspace{0.2cm}=\qhyp43{a^2,q^2b^2,ab,qab}{-qab,-q^2ab,a^2b^2}{q^2,t^2}+\frac{bt(1-a^2)}{(1+ab)(1+qab)}\qhyp43{q^2a^2,q^2b^2,qab,q^2ab}{-q^2ab,-q^3ab,q^2a^2b^2}{q^2,t^2}.
\end{eqnarray}
\end{thm}
\begin{proof}
Start with the Ismail--Wilson product generating function for Askey--Wilson polynomials \eqref{AWgf}, and replace
$(w,a,b,c,d)\mapsto({\mathrm i},{\mathrm i}a,{\mathrm i}b,-{\mathrm i}a,-{\mathrm i}qb)$
using Theorem \ref{thm318}. Finally splitting into even and odd terms using \eqref{eqthm318} and simplifying completes the proof.
\end{proof}

\noindent By starting with the Ismail--Wilson generating function for Askey--Wilson polynomials \eqref{AWgf}, substituting parameters and $w$ accordingly using Theorem 
\ref{thm321xnew}, we produce the following summation for a non-terminating ${}_2\phi_1$.
\begin{thm}
Let $0<|q|<1$, $a,b,t\in\CCast$ such that $|t|<1$. Then
\begin{eqnarray}\label{eq:thm45}
&&\hspace{-2.5cm}\qhyp21{qa^2,qab}{ab}{q,t}=
\frac{(q^2a^2t;q)_\infty}{(t;q)_\infty}\left(1+\frac{(1-\frac{b}{qa})qa^2t}{(1-ab)}\right)\nonumber\\
&&\hspace{-1.3cm}=\frac{(q^2a^2t;q^2)_\infty}{(qt;q^2)_\infty}\qhyp32{qa^2,q^3a^2,q^4ab}{q^2,ab}{q^4,t^2}\nonumber\\
&&\hspace{-0.3cm}+\frac{(q^2a^2t;q^2)_\infty\,(1-qa^2)(1-q^2ab)t}{(qt;q^2)_\infty\,(1-q^2)(1-ab)}\qhyp32{q^3a^2,q^5a^2,q^6ab}{q^6,q^2ab}{q^4,t^2}.
\end{eqnarray}
\end{thm}
\begin{proof}
Start with the Ismail--Wilson generating function for Askey--Wilson polynomials \eqref{AWgf} and replace 
\[
(a,b,c,d)\mapsto(\ppm q^\frac12,a,b).
\]
and first $w=qa$ and then $w=1/(qa)$. This produces two separate products of ${}_2\phi_1$'s and in one of them both of them can be summed, one because it is terminating and the other can be written as the $q$-binomial theorem with base $q^2$. In the other, one of the ${}_2\phi_1$'s is summable using the $q$-binomial theorem with base $q^2$.
Then splitting the infinite series into even and odd parts, utilizing Theorem \ref{thm321xnew} one obtains a sum of two ${}_6\phi_5$'s with base $q^2$ which can then be written as a sum of two ${}_3\phi_2$'s with base $q^4$. Finally, simplifying, the nonterminating summation follows.
\end{proof}

\begin{rem}
The first equality in \eqref{eq:thm45}\ is not deep
as the evaluation is easily seen to follow directly from contiguous relations. Indeed in the $_2\phi_1$ series there is a pair of
corresponding upper and lower parameters, namely $qab$ and $ab$, that multiplicatively only differ by the base $q$. This explains that the
left-hand side can be expanded in two terms, each being a summable
$_1\phi_0$ series. The two $_3\phi_2$ series also have similar pairs of
corresponding contiguous upper and lower parameters.
That means that the second expression on the right-hand side essentially consists of 4 terms which are all $_2\phi_1$ series.
The second equality is much more interesting than the first, since the just mentioned $_2\phi_1$ series are not summable by the $q$-Gauss summation.
\end{rem}
\noindent By starting with Theorem \ref{thmWeiWang} and utilizing the Ismail--Wilson generating function for Askey--Wilson polynomials we obtain the following transformation formula.
\begin{thm}
Let $0<|q|<1$, $a,b\in\CCast$, $|t|<1$. Then
\begin{eqnarray}
&&\hspace{-0.5cm}\qhyp21{a,q^\frac12a}{a^2}{q,t}
\qhyp21{b,q^\frac12b}{qb^2}{q^\frac12t}=
\qhyp87{\ppm\sqrt{ab},\ppm q^\frac14\sqrt{ab},\ppm q^\frac12\sqrt{ab},\ppm q^\frac34\sqrt{ab}}{q^\frac12,-a,-q^\frac12a,-q^\frac12b,-qb,ab,q^\frac12ab}{q,t^2}\nonumber\\
&&\hspace{3cm}+{}
\frac{(1-q^\frac12ab)t}{(1-q^\frac12)(1+a)(1+q^\frac12b)}\qhyp76{q^\frac32ab,q^2ab,q^\frac52ab,0,0,0,0}{q^\frac32,-q^\frac12a,-qa,-qb,-q^\frac32b,q^\frac12ab}{q,t^2}.
\end{eqnarray}
\end{thm}
\begin{proof}
Start with the Ismail--Wilson generating function for Askey--Wilson polynomials \eqref{AWgf} and replace 
\[
(z,a,b,c,d)\mapsto(q^\frac14,q^{\ppm\frac14}a,q^{-\frac14}b,q^{-\frac34}b).
\]
Then splitting the infinite series into even and odd parts, utilizing Theorem \ref{thmWeiWang} and simplifying, the transformation follows.
\end{proof}

\noindent This transformation has an interesting $q\to 1^{-}$ limit
\begin{eqnarray}
&&\hspace{-0.6cm}\hyp21{a,a+\frac12}{2a}{t}\hyp21{b,b+\frac12}{2b+1}{t}\nonumber\\
&&\hspace{-0.3cm}=\hyp43{\frac{a+b}{2},\frac{a+b}{2}+\frac14,\frac{a+b+1}{2},\frac{a+b}{2}+\frac34}{\frac12,a+b,a+b+\frac12}{t^2}+\left(\tfrac{a+b}{2}+\tfrac14\right)t\hyp32{a+b+\frac32,a+b+2,a+b+\frac52}{\frac32,a+b+\frac12}{\frac{t^2}{16}}.
\end{eqnarray}

\medskip 
\noindent Note that integral representations 
for the product of two nonterminating ${}_2\phi_1$'s appearing 
Theorems \ref{thm515}--\ref{thm518} can be computed in the same manner as was computed above. However, we leave these as exercises for the reader.

In \cite{Schlosser2018}, \eqref{Andrews} was used to prove the product formula
\cite[Theorem~4]{Schlosser2018} (which we reproduce in \eqref{firsteqM} below). We give an independent proof and are also provide an integral representation for this product.

\begin{thm}
Let {$|q|<1$}, $f,a,b,z\in\CCast$ with $|z|<1$, $w=\expe^{i\psi}$, $\sigma\in(0,\infty)$ such that in the integrand
the denominator elements of the infinite $q$-shifted factorials have modulus less than unity. Then 
\begin{eqnarray}
&&\hspace{-0.5cm}\qhyp21{a,\frac{q}{a}}{-q}{q,z}
\qhyp21{b,\frac{q}{b}}{-q}{q,-z}\nonumber\\
&&\hspace{0.1cm}=\qhyp43{ab,\frac{q^2}{ab},\frac{qa}{b},\frac{qb}{a}}{-q^2,\ppm q}{q^2,z^2}+
\frac{(b-a)(1-\frac{q}{ab})z}{1-q^2}\qhyp43{qab,\frac{q^3}{ab},\frac{q^2a}{b},\frac{q^2b}{a}}{-q^2,\ppm q^3}{q^2,z^2}\label{firsteqM}\\
&&\hspace{0.1cm}
=\frac{1}{2\pi}\frac{(q,\ppm a,\ppm b,\ppm\frac{q}{b};q)_\infty}{\vartheta(\ppm f;q)(-q,ab,\frac{qa}{b};q)_\infty}
\int_{-\pi}^\pi 
\frac{(\ppm{\mathrm i}f\frac{\sigma}{w},(\ppm{\mathrm i}f,{\mathrm i}qa)\frac{w}{\sigma};q)_\infty}
{(\ppm{\mathrm i}\frac{\sigma}{w},(-{\mathrm i}a,{\mathrm i}b,{\mathrm i}\frac{q}{b})\frac{w}{\sigma};q)_\infty}
\qhyp32{ab,\frac{qa}{b},-\frac{{\mathrm i}q}{a}\frac{\sigma}{w}}{-q,{\mathrm i}qa\frac{w}{\sigma}}{q,{\mathrm i}z\frac{w}{\sigma}}\dd\psi.
\label{intrep1}
\end{eqnarray}
\end{thm}
\begin{proof}
First utilize the Ismail--Wilson generating function for Askey--Wilson polynomials \eqref{AWgf} with
\[
(w,a,b,c,d,t)\mapsto({\mathrm i},-{\mathrm i}a,-\tfrac{{\mathrm i}q}{a},{\mathrm i}b,\tfrac{{\mathrm i}q}{b},{\mathrm i}z).
\]
\noindent An application of  Andrews' $q$-analogue of Whipple's ${}_3F_2(1)$ summation \eqref{Andrews}, followed by $(c,e)\mapsto(a,ab)$ and splitting the sum into the even and odd indexed parts completes the proof of \eqref{firsteqM}.
Finally, apply \cite[Theorem 3.3]{CohlCostasSantos23} to \eqref{Andrews}
with 
$(a,b,c,d,t,z)\mapsto
\big(-{\mathrm i}a,{\mathrm i}b,\tfrac{{\mathrm i}q}{b},-\tfrac{{\mathrm i}q}{a},{\mathrm i}z,{\mathrm i}\big)$, {which produces \eqref{intrep1},}
 completing the proof.
\end{proof}

\begin{rem}
The above transformation can also be obtained by starting with the Ismail--Wilson generating function \eqref{AWgf} and applying Theorem \ref{thm320}.
\end{rem}

\noindent A classic nonterminating $q$-product formula which generalizes T.~Clausen's (1828) formula
\cite[\href{http://dlmf.nist.gov/16.12.E2}{(16.12.2)}]{NIST:DLMF}, \cite{Clausen1828}
\begin{equation*}
\left(\hyp21{a,b}{a+b+\frac12}{z}\right)^2=\hyp32{2a,2b,a+b}{a+b+\frac12,2a+2b}{z},    
\end{equation*}
was derived by F.~H.~Jackson (1940) \cite{Jackson1940}, namely 
\begin{equation*}
\qhyp21{a^2,b^2}{qa^2b^2}{q^2,z}
\qhyp21{a^2,b^2}{qa^2b^2}{q^2,qz}
=\qhyp43{a^2,b^2,\ppm ab}{a^2b^2,\ppm q^\frac12 ab}{q,z},
\end{equation*}
where $|q|<1$ and $|z|<1$.
A $q$-analogue of a result given by Orr (1899)
\cite[(10.1.5)]{Bailey64},
\cite{Orr1899}
\begin{equation*}
\hyp21{a,b}{a+b-\frac12}{z}\hyp21{a,b}{a+b+\frac12}{z}=\hyp32{2a,2b,a+b}{2a+2b-1,a+b+\frac12}{z},
\end{equation*} 
is derived in Nassrallah's (1982) thesis, namely 
\cite[(3.3.1)]{Nassrallah}, given below in \eqref{eq:N1}.
For this product formula we provide an integral representation.
\begin{thm}
Let $0<|q|<1$, $f,a,b,z\in\CCast$ with $|z|<1$, $w=\expe^{i\psi}$, $\sigma\in(0,\infty)$ such that in the integrand 
the denominator elements of the infinite $q$-shifted factorials have modulus less than unity. Then 
\begin{eqnarray}\label{eq:N1}
&&\hspace{-1.0cm}\qhyp21{a^2,b^2}{\frac{a^2b^2}{q}}{q^2,z}
\qhyp21{a^2,b^2}{qa^2b^2}{q^2,qz}=\qhyp43{a^2,b^2,\ppm ab }{\frac{a^2b^2}{q},\ppm q^\frac12ab}{q,z}\\
&&\hspace{-0.5cm}
   =\frac{1}{2\pi}\frac{(q^2,a^2,a^2,qa^2,\frac{a^2}{q},b^2,qb^2;q^2)_\infty}{(f,\frac{q^2}{f},qf,\frac{q}{f},a^4,a^2b^2,qa^2b^2;q^2)_\infty}
   \nonumber\\
&&\hspace{0.5cm}
\times
\int_{-\pi}^\pi 
\frac{((q^{\frac12}f,\frac{q^\frac32}{f})\frac{\sigma}{w},(q^{\frac12}f,\frac{q^\frac32}{f},q^\frac12 a^4b^2)\frac{w}{\sigma};q^2)_\infty}{(q^{\ppm \frac12}\frac{\sigma}{w},(q^{\ppm\frac12}a^2,q^{-\frac12}b^2)\frac{w}{\sigma};q^2)_\infty}
\qhyp32{a^4,a^2b^2,q^{\frac12}b^2\frac{\sigma}{w}}{qa^2b^2,q^{-\frac12}a^4b^2\frac{w}{\sigma}}{q^2,q^{\frac12}z\frac{w}{\sigma}}\dd\psi.\label{eq:irN1}
\end{eqnarray}
\end{thm}
\begin{proof}
Start with 
\cite[(3.3.1)]{Nassrallah} which gives \eqref{eq:N1}.
Then application of \cite[Theorem 3.3]{CohlCostasSantos23} with 
\begin{equation*}
(a,b,c,d,t,z,q)\mapsto
\big(q^\frac12a^2,q^{-\frac12}a^2,q^{\frac12}b^2,q^{-\frac12}b^2,q^{\frac12}z,
q^{\frac12},q^2\big)
\end{equation*}
yields the integral representation \eqref{eq:irN1},
completing the proof.
\end{proof}

\noindent A $q$-analogue of another
result given by
Orr (1899)
\cite[(10.1.6)]{Bailey64},
\cite{Orr1899}
\begin{equation*}
\hyp21{a,b}{a+b-\frac12}{z}\hyp21{a,b-1}{a+b-\frac12}{z}=\hyp32{2a,2b-1,a+b-1}{2a+2b-2,a+b-\frac12}{z},
\end{equation*} 
is derived in Nassrallah's (1982) thesis, namely 
\cite[(3.3.2)]{Nassrallah} (which we have corrected), that we give below
  in \eqref{NassrallahB}. We also provide an integral representation for this
product formula.

\begin{thm}
Let $0<|q|<1$, $f,a,b,z\in\CCast$ with $|z|<1$, $w=\expe^{i\psi}$, $\sigma\in(0,\infty)$ such that in the integrand
the denominator elements of the infinite $q$-shifted factorials have modulus less than unity. Then 
\begin{eqnarray}\label{NassrallahB}
&&\hspace{-1.0cm}\qhyp21{qa^2,qb^2}{qa^2b^2}{q^2,z}
\qhyp21{\frac{a^2}{q},qb^2}{qa^2b^2}{q^2,qz}=\qhyp43{a^2,qb^2,\ppm ab }{a^2b^2,\ppm q^\frac12ab}{q,z}\\
&&\hspace{-0.5cm}
=\frac{1}{2\pi}\frac{(q^2,qa^2,a^2,a^2,\frac{a^2}{q},qb^2,q^2b^2;q^2)_\infty}{(f,\frac{q^2}{f},qf,\frac{q}{f},a^4,qa^2b^2,q^2a^2b^2;q^2)_\infty}
\nonumber\\
&&\hspace{0.5cm}
\times
\int_{-\pi}^\pi 
\frac{((q^{\frac12}f,\frac{q^\frac32}{f})\frac{\sigma}{w},(q^{\frac12}f,\frac{q^\frac32}{f},q^\frac32a^4b^2)\frac{w}{\sigma};q^2)_\infty}{(q^{\ppm \frac12}\frac{\sigma}{w},(q^{\ppm\frac12}a^2,q^\frac32b^2)\frac{w}{\sigma};q^2)_\infty}
\qhyp32{a^4,q^2a^2b^2,q^{\frac12}b^2\frac{\sigma}{w}}{qa^2b^2,q^{\frac32}a^4b^2\frac{w}{\sigma}}{q^2,q^{\frac12}z\frac{w}{\sigma}}\dd\psi.\label{eq:irN2}
\end{eqnarray}
\end{thm}
\begin{proof}
Start with the corrected version of 
\cite[(3.3.2)]{Nassrallah}.
Then application of \cite[Theorem 3.3]{CohlCostasSantos23} with 
\begin{equation*}
(a,b,c,d,t,z,q)\mapsto
\big(q^{\frac12}a^2,q^{-\frac12}a^2,q^{\frac12}b^2,q^{\frac32}b^2,q^{\frac12}z,
  q^{\frac12},q^2\big)
\end{equation*}
yields the integral representation \eqref{eq:irN2}, completing the proof.
\end{proof}


\begin{rem}
\label{rem34}
By starting with the the product generating function for Askey--Wilson polynomials \eqref{AWgf}
where $x=\frac12(w+w^{-1})$,
and replacing
\begin{equation*}    
(w,a,b,c,d)\mapsto\left(q^{-\frac14}, q^{\frac14}a,q^{-\frac14}\tfrac{b}{a},q^{\frac14}\tfrac{b}{a},q^{-\frac14}a\right),
\end{equation*}
followed by using the basic hypergeometric representation for Askey--Wilson polynomials \eqref{aw:def1}, the balanced ${}_4\phi_3$ is converted into the
form of the 
second $q$-analogue of Bailey's ${}_4F_3(1)$ sum, namely \eqref{eq:BIS-2.1} \cite[\href{http://dlmf.nist.gov/17.7.E13}{(17.7.13)}]{NIST:DLMF},
where $q^2$ is replaced with $q$. Then adopting the right-hand side of \eqref{eq:BIS-2.1} and replacing  $q\mapsto q^2$ in the entire expression and setting $t\mapsto q^\frac14z$,
followed by replacing  $(a,b)\mapsto(qb^2,qa^2)$ 
produces \eqref{NassrallahB}.
\end{rem}

\begin{rem}
As pointed out to us by Slobodan Damjanovic, the $q$-quadratic summation in \eqref{eq:N1}
is equivalent (by comparison of the coefficients of $z^n$ on both sides of the identity, etc.)
to Guo's summation in \eqref{eq:g}.
Similarly, the $q$-quadratic summation in \eqref{NassrallahB}
is equivalent (by comparison of the coefficients of $z^n$ on both sides of the identity, etc.)
to Wei and Wang's summation in \eqref{eq:wx}.
\end{rem}

\begin{rem}
If one sets $a^2=q$ in \eqref{NassrallahB} or $b^2=q$ in \eqref{eq:thm21eq},
then one obtains
\begin{equation*}
\qhyp21{q^2,a^2}{qa^2}{q^2,z}=\qhyp32{q,\ppm a}{\ppm q^\frac12 a}{q,z}
\end{equation*}
(the series being convergent for $|q|<1$ and $|z|<1$).
Note however that this result is trivially satisfied,
as the identity holds termwise.
\end{rem}

\noindent Another formula relating a product of two nonterminating
  ${}_2\phi_1$ series with  a nonterminating ${}_4\phi_3$ series
was obtained by Srivastava and Jain \cite[(4.9)]{SrivJain1986} (see also
\cite[Theorem 2.1]{Schlosser2018}), given in \eqref{eq:JS} below. For this
product formula we also provide an integral representation.
\begin{thm}
\label{thm45}
Let $|q|<1$, $f,a,b,z\in\CCast$ with $|z|<1$, $w=\expe^{i\psi}$, $\sigma\in(0,\infty)$ such that in the integrand, 
the denominator elements of the infinite $q$-shifted factorials have modulus less than unity. Then 
\begin{eqnarray}\label{eq:JS}
&&\hspace{-0.8cm}\qhyp21{\ppm a}{a^2}{q,z}
\qhyp21{\ppm b}{b^2}{q,-z}=\qhyp43{\ppm ab,\ppm qab}{qa^2,qb^2,a^2b^2}{q^2,z^2}\\
&&\hspace{-0.5cm}
=\frac{1}{2\pi}\frac{(q,\ppm a,\ppm b,\ppm b;q)_\infty}{\vartheta(\ppm f;q)(\ppm ab,b^2;q)_\infty}
\nonumber\\
&&\hspace{0.5cm}
\times
\int_{-\pi}^\pi 
\frac{(({\mathrm i}f,-{\mathrm i}\frac{q}{f})\frac{\sigma}{w},({\mathrm i}f,-{\mathrm i}\frac{q}{f},-{\mathrm i}ab^2)\frac{w}{\sigma};q)_\infty}{(\ppm{\mathrm i}\frac{\sigma}{w},(-{\mathrm i}a,\ppm{\mathrm i}b)\frac{w}{\sigma};q)_\infty}
\qhyp32{\ppm ab,{\mathrm i}a\frac{\sigma}{w}}{a^2,-{\mathrm i}ab^2\frac{w}{\sigma}}{q,{\mathrm i}z\frac{w}{\sigma}}\dd\psi.\label{eq:irJS}
\end{eqnarray}
\end{thm}
\begin{proof}
Start with \eqref{Andrews2}. Application of  \cite[Theorem 3.3]{CohlCostasSantos23} with
\begin{equation*}
(a,b,c,d,t,z)\mapsto(-{\mathrm i}a,{\mathrm i}b,-{\mathrm i}b,{\mathrm i}a,{\mathrm i}z,{\mathrm i})
\end{equation*}
yields \eqref{eq:irJS},
completing the proof.
\end{proof}
\noindent This formula is a $q$-analogue of Bailey's formula \cite[(2.11)]{Bailey1928}
\begin{equation*}
\hyp11{a}{2a}{z}
\hyp11{b}{2b}{-z}=
\hyp23{\frac12(a+b),\frac12(a+b+1)}
{a+\frac12,b+\frac12,a+b}{\tfrac14z^2}.
\end{equation*}
Along the lines of \cite{Schlosser2018}, we are able to derive similar product formulas. The first is given as follows.

\begin{thm}
\label{thm21}
Let $0<|q|<1$, $f,a,b,z\in\CCast$ with $|z|<1$, $w=\expe^{i\psi}$, $\sigma\in(0,\infty)$ such that in the integrand
the denominator elements of the infinite $q$-shifted factorials have modulus less than unity. Then 
\begin{eqnarray}
\label{thm21eq}
&&\hspace{-1.0cm}\qhyp21{qa^2,qb^2}{qa^2b^2}{q^2,z}
\qhyp21{\frac{a^2}{q},\frac{b^2}{q}}{\frac{a^2b^2}{q}}{q^2,qz}=\qhyp43{a^2,b^2,\ppm ab }{\frac{a^2b^2}{q},\ppm q^\frac12ab}{q,z}\label{eq:thm21eq}\\
&&\hspace{-0.5cm}
=\frac{1}{2\pi}\frac{(q^2,qa^2,a^2,a^2,\frac{a^2}{q},b^2,\frac{b^2}{q};q^2)_\infty}{(f,\frac{q^2}{f},qf,\frac{q}{f},a^4,a^2b^2,\frac{a^2b^2}{q};q^2)_\infty}
\nonumber\\
&&\hspace{0.5cm}
\times
\int_{-\pi}^\pi 
\frac{((q^{\frac12}f,\frac{q^\frac32}{f})\frac{\sigma}{w},(q^{\frac12}f,\frac{q^\frac32}{f},q^{-\frac12}a^4b^2)\frac{w}{\sigma};q^2)_\infty}{(q^{\ppm\frac12}\frac{\sigma}{w},(q^{\ppm\frac12}a^2,q^{-\frac12}b^2)\frac{w}{\sigma};q^2)_\infty}
\qhyp32{a^4,a^2b^2,q^{\frac12}b^2\frac{\sigma}{w}}{qa^2b^2,q^{-\frac12}a^4b^2\frac{w}{\sigma}}{q^2,q^{\frac12}z\frac{w}{\sigma}}\dd\psi.
\label{thm21eqi}
\end{eqnarray}
\end{thm}
\begin{proof}
First start with the Ismail--Wilson product generating function for Askey--Wilson polynomials \eqref{AWgf}, and apply
\begin{equation*}    
(w,a,b,c,d)\mapsto\big(q^{-\frac14},q^{\frac14}a,q^{-\frac14}\tfrac{b}{a},q^{\frac14}\tfrac{b}{a},q^{\frac34}a\big).
\end{equation*}
Then using the basic hypergeometric representation for Askey--Wilson polynomials \eqref{aw:def1}, the balanced ${}_4\phi_3$ is converted into the
form of the 
first $q$-analogue of Bailey's ${}_4F_3(1)$ sum, namely \cite[\href{http://dlmf.nist.gov/17.7.E12}{(17.7.12)}]{NIST:DLMF}
\begin{equation}
\qhyp43{q^{-2n},q^{2n}b^2,a,qa}{b,qb,q^2a^2}{q^2,q^2}=a^n\frac{(-q,\frac{b}{a};q)_n}{(-qa,b;q)_n},
\label{sum1}
\end{equation}
where $q^2$ is replaced with $q$. Then adopting the right-hand side of \eqref{sum1} and replacing  $q\mapsto q^2$ in the entire expression and setting $t\mapsto q^\frac14z$, followed by replacing $(a,b)\mapsto(qa^2,qb^2)$,
product formula \eqref{eq:thm21eq} is obtained.
Replace $q^2\mapsto q$. Then the application of
\cite[Theorem 3.3]{CohlCostasSantos23} with
\begin{equation*}
(a,b,c,d,t,z)\mapsto
\big(q^\frac14a^2,q^{-\frac14}a^2,q^\frac14b^2,q^{-\frac14}b^2,q^\frac14z,q^\frac14\big),
\end{equation*}
followed by replacing $q\mapsto q^2$,yields the integral
representation \eqref{thm21eq}, completing the proof.
\end{proof}

\begin{cor}
Let $a,b,z\in\CC$ such that $|z|<1$ and $2a+2b\not\in-\N$. Then
\begin{eqnarray}
&&\hyp21{a,b}{a+b+\frac12}{z}\hyp21{a+1,b+1}{a+b+\frac32}{z}=\hyp32{2a+1,2b+1,a+b+1}{2a+2b+1,a+b+\frac32}{z}
.\end{eqnarray}
\end{cor}
\begin{proof}
Start with Theorem \ref{thm21} and replace $(a,b)\mapsto(q^a,q^b)$. Then
take the limit as $(q^2,q)\to 1^{-}$ on the left and right-hand
sides, respectively, followed by the replacement
$(a,b)\mapsto (a+\frac12,b+\frac12)$, completing the proof.
\end{proof}

\begin{rem}
If one starts with \eqref{AWgf} and applies the substitution
$(w,a,b,c,d)\mapsto({\mathrm i},a,b,-a,-b)$ (notice that $w=i$ or $w=-{\mathrm i}$
is equivalent to $x=0$), then the Askey--Wilson polynomial
appearing in the series, written as a ${}_4\phi_3$ using \eqref{aw:def3},
can be summed using \eqref{newsum3}. Using the fact that only those
terms of the series survive for which the summation index $n$ is even
(the other terms vanish), replacing $n$ by $2n$, followed by a series of
simplifications, one arrives at the identity 
\begin{eqnarray*}
&&\hspace{-4.4cm}\qhyp21{a,b}{-ab}{q,t}\qhyp21{a,b}{-ab}{q,-t}=\qhyp43{ab,qab,a^2,b^2}{-ab,-qab,a^2b^2}{q^2,t^2},
\end{eqnarray*}
where $|q|<1$ and $|t|<1$ for convergence.
This product formula was previously derived in Srivastava (1987)
\cite[(3.13)]{Srivastava87} (see also \cite[(3.1)]{Schlosser2018}).
\end{rem}

\noindent
By starting with Theorem \ref{thm320new} and inserting it into the Ismail--Wilson generating function \eqref{AWgf} we obtain the following $q$-quadratic transformation for a product of two ${}_2\phi_1$s given in terms of a sum of two $q$-quadratic ${}_4\phi_3$'s.
\begin{thm}
Let $0<|q|<1$, $a,b,z\in\CCast$ with $|z|<1$. Then
\begin{eqnarray}
&&\hspace{-1.1cm}\qhyp21{\ppm a}{\frac{a^2}{q}}{q,z}\qhyp21{\ppm b}{qb^2}{q,-qz}
\nonumber\\
&&\hspace{0.5cm}
=\qhyp43{\ppm ab,\ppm qab}{\frac{a^2}{q},qb^2,a^2b^2}{q^2,z^2}+\frac{z(1-a^2b^2)}{(1-\frac{a^2}{q})(1-qb^2)}\qhyp43{\ppm qab,\ppm q^2ab}{qa^2,q^3b^2,a^2b^2}{q^2,z^2}.
\end{eqnarray}
\end{thm}
\begin{proof}
First start with the Ismail--Wilson product generating function for Askey--Wilson polynomials \eqref{AWgf}, and replace
\begin{equation*}    
(w,a,b,c,d)\mapsto({\mathrm i}q^{\frac12},\ppm{\mathrm i}a,\ppm{\mathrm i}b).
\end{equation*}
Then using the Askey--Wilson summation formula  Theorem \ref{thm320new}, splitting the infinite sum into even and odd contributions and simplifying completes the proof.
\end{proof}

\noindent 
By starting with Theorem \ref{thm321new} and inserting it into the Ismail--Wilson generating function \eqref{AWgf} we obtain the following $q$-quadratic transformation for a product of two ${}_2\phi_1$s given in terms of a sum of two $q$-quadratic ${}_5\phi_4$'s.
\begin{thm}
Let $0<|q|<1$, $a,b,z\in\CCast$ with $|z|<1$. Then
\begin{eqnarray}
&&\hspace{-0.7cm}\qhyp21{\ppm qa}{qa^2}{q,z}\qhyp21{b,-qb}{q^2b^2}{q,-qz}\nonumber\\
&&\hspace{0.0cm}=\qhyp54{q^3b,\ppm qab,\ppm q^2ab}{qa^2,qb,q^3b^2,q^2a^2b^2}{q^2,z^2}+\frac{z(1-q^2a^2b)}{(1-qa^2)(1-qb)}\qhyp54{\ppm q^2ab,\ppm q^3ab,q^4a^2b}{q^3a^2,q^3b^2,q^2a^2b,q^4a^2b^2}{q^2,z^2}.
\end{eqnarray}
\end{thm}
\begin{proof}
First start with the Ismail--Wilson product generating function for Askey--Wilson polynomials \eqref{AWgf}, and replace
\begin{equation*}    
(w,a,b,c,d)\mapsto({\mathrm i}q^{\frac12},\ppm{\mathrm i}a,{\mathrm i}b,-{\mathrm i}qb).
\end{equation*}
Then using the Askey--Wilson summation formula  Theorem \ref{thm321new}, splitting the infinite sum into even and odd contributions and simplifying completes the proof.
\end{proof}

\noindent By starting with Theorem \ref{thmAW38} and inserting it into the Ismail--Wilson generating function \eqref{AWgf} we obtain the following transformation for a product of two ${}_2\phi_1$'s given in terms of a sum of a ${}_6\phi_5$'s and a ${}_4\phi_3$ each with base $q^2$.
\begin{thm}
Let $0<|q|<1$, $a,b,z\in\CCast$ with $|z|<1$. Then
\begin{eqnarray}
&&\hspace{-0.3cm}\qhyp21{a,-qa}{a^2}{q,z}\qhyp21{b,-qb}{q^2b^2}{q,-qz}=\frac{(1+\frac{q}{a})(1-qab)}{q(1+\frac{1}{a})(1-ab)}\qhyp65{q,q^3b,ab,-qab,-q^2ab,q^3ab}{q,qb,qa^2,q^3b^2,q^2a^2b^2}{q^2,z^2}\nonumber\\
&&\hspace{6.6cm}
-\frac{(1-q)(1+qb)}{q(1+\frac{1}{a})(1-ab)}\qhyp65{q^3,q^4b^2,ab,\ppm qab,-q^2ab}{q,qa^2,q^2b^2,q^3b^2,q^2a^2b^2}{q^2,z^2}\nonumber\\
&&\hspace{6.6cm}+\frac{(1+qab)z}{(1+a)(1-qb)}\qhyp43{qab,\ppm q^2ab,-q^3ab}{qa^2,q^3b^2,q^2a^2b^2}{q^2,z^2}.
\end{eqnarray}
\end{thm}
\begin{proof}
First start with the Ismail--Wilson product generating function for Askey--Wilson polynomials \eqref{AWgf}, and replace
\begin{equation*}    
(w,a,b,c,d)\mapsto({\mathrm i}q^{\frac12},{\mathrm i}a,-{\mathrm i}qa,{\mathrm i}b,-{\mathrm i}qb).
\end{equation*}
Then using the Askey--Wilson summation formula  Theorem \ref{thmAW38}, splitting the infinite sum into even and odd contributions, scaling the variables appropriately and simplifying completes the proof.
\end{proof}

\subsection{Application of Cayley--Orr type expansion formulas}

\noindent By using Cayley--Orr type expansion formulas
stated in Gasper and Rahman's textbook as \cite[Exercises 3.17--3.19]{GaspRah}
(obtained by Singh~\cite{Singh59} and Nassrallah \cite{Nassrallah}),
one may verify the product formulas we have derived
and derive new ones. We start with
\cite[Exercise 3.18]{GaspRah} where we have replaced $(a,b)\mapsto(\frac{a}{q},\frac{b}{q})$.

\begin{lem}
\label{firstprop}
Let $0<|q|<1$, $a,b,c,z\in\CCast$, $|z|<1$, $|q^2cz|<|ab|$ and 
\begin{equation}
\frac{(\frac{q^3cz}{ab};q^2)_\infty}{(z;q^2)_\infty}\qhyp21{\frac{a}{q},\frac{b}{q}}{c}{q,\frac{q^2cz}{ab}}
=\sum_{n=0}^\infty a_nz^n. 
\label{ex318prop}
\end{equation}
Then
\begin{equation}
\qhyp21{\frac{q^2c}{a},\frac{q^2c}{b}}{q^2 c}{q^2,z}
\qhyp21{\frac{a}{q},\frac{b}{q}}{c}{q^2,\frac{q^2cz}{ab}}=\sum_{n=0}^\infty
\frac{(qc;q^2)_n}{(q^2c;q^2)_n}a_nz^n.
\label{ex318}
\end{equation}
\end{lem}

\begin{thm}
Theorem \ref{thm21} is verified by Lemma \ref{firstprop} with $c=\frac{ab}{q}$.
\end{thm}
\begin{proof}
If one takes $c=ab/q$ then the left-hand side of
\eqref{ex318} becomes the left-hand side of \eqref{thm21eq}. Replacing $c=ab/q$ in \eqref{ex318prop} produces
\begin{equation*}
\frac{(q^2z;q^2)_\infty}{(z;q^2)_\infty}
\qhyp21{\frac{a}{q},\frac{b}{q}}{\frac{ab}{q}}{q,qz}=\frac{1}{1-z}\qhyp21{\frac{a}{q},\frac{b}{q}}{\frac{ab}{q}}{q,qz}=\sum_{n=0}^\infty a_nz^n.
\end{equation*}
By expanding the geometric series $(1-z)^{-1}$ then one can see that from \eqref{ex318prop}, using the $q$-Pfaff--Saalsch\"utz sum \cite[(II.12)]{GaspRah}, that one obtains
\begin{equation*}
a_n=\sum_{k=0}^n \frac{(\frac{a}{q},\frac{b}{q};q)_k}{(q,\frac{ab}{q};q)_k}q^k=
\frac{(a,b;q)_n}{(q,\frac{ab}{q};q)_n}.
\end{equation*}
Therefore from \eqref{ex318}, we see that 
\begin{eqnarray*}
&&\hspace{-3.5cm}\qhyp21{qa,qb}{qab}{q^2,z}
\qhyp21{\frac{a}{q},\frac{b}{q}}{\frac{ab}{q}}{q^2,qz}=\sum_{n=0}^\infty 
\frac{(ab;q^2)_n}{(qab;q^2)_n}
\frac{(a,b;q)_n}{(q,\frac{ab}{q};q)_n}
z^n,
\end{eqnarray*}
which completes the verification using
\eqref{sqPoch}. 
\end{proof}

\noindent Next we turn to \cite[Exercise 3.19]{GaspRah}.

\begin{lem}
\label{secondprop}
Let $0<|q|<1$, $a,b,c,z\in\CCast$, $|z|<1$, $|q^2cz|<|ab|$ and 
\begin{equation}
\frac{(\frac{qcz}{ab};q^2)_\infty}{(z;q^2)_\infty}\qhyp21{\frac{a}{q},b}{\frac{c}{q}}{q,\frac{cz}{ab}}
=\sum_{n=0}^\infty a_nz^n. 
\label{ex319prop}
\end{equation}
Then
\begin{equation}
\qhyp21{\frac{qc}{a},\frac{c}{qb}}{c}{q^2,z}
\qhyp21{a,b}{c}{q^2,\frac{cz}{ab}}=\sum_{n=0}^\infty
\frac{(\frac{c}{q};q^2)_n}{(c;q^2)_n}a_nz^n.
\label{ex319}
\end{equation}
\end{lem}

\begin{thm}
Nassrallah's second result \eqref{NassrallahB} is verified by Lemma \ref{secondprop} with $c=qab$.
\label{thm39}
\end{thm}

\begin{proof}
If one takes  $(a,b,c)\mapsto(qb,\frac{a}{q},qab)$
in Lemma \ref{secondprop},
then the left-hand side of
\eqref{ex319} becomes the left-hand side of \eqref{ex319prop}. Replacing $c=qab$ in \eqref{ex319prop} produces
\begin{equation*}
\frac{(q^2z;q^2)_\infty}{(z;q^2)_\infty}
\qhyp21{\frac{a}{q},b}{ab}{q,qz}=\frac{1}{1-z}\qhyp21{\frac{a}{q},b}{ab}{q,qz}=\sum_{n=0}^\infty a_nz^n.
\end{equation*}
By expanding the geometric series $(1-z)^{-1}$ then one can see that from \eqref{ex319prop}, using the $q$-Pfaff--Saalsch\"utz sum \cite[(II.12)]{GaspRah}, that one obtains
\begin{equation*}
a_n=\sum_{k=0}^n \frac{(\frac{a}{q},b;q)_k}{(q,ab;q)_k}q^k=
\frac{(a,qb;q)_n}{(q,ab;q)_n}.
\end{equation*}
Therefore from \eqref{ex319}, we see that 
\begin{eqnarray*}
&&\hspace{-3.5cm}\qhyp21{qa,qb}{qab}{q^2,z}\qhyp21{\frac{a}{q},qb}{qab}{q^2,qz}=\sum_{n=0}^\infty 
\frac{(ab;q^2)_n}{(qab;q^2)_n}
\frac{(a,qb;q)_n}{(q,ab;q)_n}
z^n,
\end{eqnarray*}
which completes the verification using
\eqref{sqPoch}. 
\end{proof}

\noindent Next we turn to \cite[Exercise 3.17]{GaspRah}.
\begin{lem}
\label{zeroprop}
Let $0<|q|<1$, $a,b,c,z\in\CCast$, $|z|<1$, $|cz|<|qab|$ and 
\begin{equation}
\frac{(\frac{cz}{ab};q^2)_\infty}{(z;q^2)_\infty}\qhyp21{a,b}{c}{q,\frac{cz}{qab}}
=\sum_{n=0}^\infty a_nz^n. 
\label{ex317prop}
\end{equation}
Then
\begin{equation}
\qhyp21{\frac{c}{a},\frac{c}{b}}{q c}{q^2,z}
\qhyp21{a,b}{qc}{q^2,\frac{cz}{qab}}=\sum_{n=0}^\infty
\frac{(c;q^2)_n}{(qc;q^2)_n}a_nz^n.
\label{ex317}
\end{equation}
\end{lem}

\begin{thm}
\label{Thm417}
Let $0<|q|<1$, $a,z\in\CCast$ such that $|z|<1$. Then one has the following three product formulas:
\begin{eqnarray}
&&\hspace{-2.0cm}\qhyp21{\ppm q^{\frac 12}{\sqrt{a}}}{qa}{q,z}
\qhyp21{\ppm \sqrt{a}}{qa}{q,-z}=
\sum_{n=0}^\infty\frac{(q^{1+2\lfloor\frac{n+1}2\rfloor}a^2;q^2)_{\lfloor\frac{n}2\rfloor}\,a^{(n-2\lfloor\frac{n}2\lfloor)}}
{(q^2;q^2)_{\lfloor\frac{n}2\rfloor}
(qa;q)_n}z^n\nonumber\\
&&\hspace{-0.0cm}=\qhyp43{\ppm q^{\frac 12}a,\ppm q^{\frac 32}a}{qa,q^2a,qa^2}{q^2,z^2}+\frac{az}{(1-qa)}\qhyp43{\ppm q^{\frac 32}a,\ppm q^{\frac 52}a}{q^2a,q^3a,q^3a^2}{q^2,z^2},\label{eq:thm4151}\\
&&\hspace{-2.0cm}\qhyp21{\ppm \sqrt{a}}{q^{\frac 12}a}{q,z}
\qhyp21{\ppm \sqrt{a}}{q^{\frac 12}a}{q,-q^{\frac 12}z}=
\sum_{n=0}^\infty
\frac{(q^{\frac 12};-q^{\frac 12})_n(a^2;q^2)_{n}}{(q;q)_n(q^{\frac 12}a;q)_n(-a;-q^{\frac 12})_n}z^n\nonumber\\
&&\hspace{0.0cm}=\qhyp87{\ppm \sqrt{a},\ppm q^{\frac 12}\sqrt{a},\ppm{\mathrm i}\sqrt{a},\ppm{\mathrm i}q^{\frac 12}\sqrt{a}}{-q^{\frac 12},\ppm q^\frac14\sqrt{a},\ppm q^\frac34\sqrt{a},-a,q^{\frac 12}a}{q,z^2}\nonumber\\
&&\hspace{1cm}+\frac{(1-a)z}{(1+q^{\frac 12})(1-q^{\frac 12}a)}
\qhyp87{\ppm q^{\frac 12}\sqrt{a},\ppm q\sqrt{a},\ppm{\mathrm i}q^{\frac 12}\sqrt{a},\ppm{\mathrm i}q\sqrt{a}}{-q^{\frac 32},\ppm q^\frac34\sqrt{a},\ppm q^\frac54\sqrt{a},q^{\frac 12}a,-qa}{q,z^2},\label{eq:thm4152}\\
&&\hspace{-2.4cm}\qhyp21{\ppm \sqrt{a}}{a}{q,z}
\qhyp21{\ppm q^{-\frac 12}\sqrt{a}}{a}{q,-qz}
=
\sum_{n=0}^\infty
\frac{(q^{-1+2\lfloor\frac{n+1}2\rfloor}a^2;q^2)_{\lfloor\frac{n}2\rfloor}}{(q^2;q^2)_{\lfloor\frac{n}2\rfloor}(a;q)_n}
z^n\nonumber\\
&&\hspace{0.0cm}=\qhyp43{\ppm q^{-\frac 12}a,\ppm q^{\frac 12}a}{a,qa,q^{-1}a^2}{q^2,z^2}+\frac{z}{(1-a)}\qhyp43{\ppm q^{\frac 12}a,\ppm q^{\frac 32}a}{qa,q^2a,qa^2}{q,z^2}.\label{eq:thm4153}
\end{eqnarray}
\end{thm}
\begin{proof}
We apply Lemma~\ref{zeroprop} with suitable specializations of the parameters.
Without yet specializing the parameters, we observe, that by virtue of
the nonterminating $q$-binomial theorem, the coefficients $a_n$ in \eqref{ex317prop}
can be computed as
\begin{align}\label{eq:ancoeff}
a_n&=\sum_{j+k=n}\frac{(\frac{c}{ab};q^2)_j}{(q^2;q^2)_j}\frac{(a,b;q)_k}{(q,c;q)_k}
\left(\frac{c}{qab}\right)^k=\frac{(\frac{c}{ab};q^2)_n}{(q^2;q^2)_n}\sum_{k=0}^n
\frac{(a,b;q)_k}{(q,c;q)_k}\frac{(q^{-2n};q^2)_k}{(\frac{q^{2-2n}ab}{c};q^2)_k}q^k\notag\\
&=\frac{(\frac{c}{ab};q^2)_n}{(q^2;q^2)_n}
\qhyp43{\ppm q^{-n},a,b}{c,\ppm\frac{q^{1-n}\sqrt{ab}}{\sqrt{c}}}{q,q},
\end{align}
the $_4\phi_3$ series being terminating and 2-balanced.
Comparing the above series with the terminating 2-balanced $_4\phi_3$ summations
derived in Subsection~\ref{subsec:kbal}, we find that the coefficients
$a_n$ reduce to closed form products by instances of Theorems~\ref{thm:n3},
\ref{thm:n4} and Corollary~\ref{cor:310s}
in the three different cases
$(a,b,c)\mapsto(a,-a,q^\sigma a^2)$, where $\sigma\in\{1,0,-1\}$.
The computation in these three cases are as follows:

\noindent
\underline{Case 1}: $(a,b,c)\mapsto (a,-a,q a^2)$:

By Theorem~\ref{thm:n3} we have
\begin{align*}
a_n&=\frac{(-q;q^2)_n}{(q^2;q^2)_n}a^{2(n-2\lfloor\frac n2\rfloor)}
\frac{(q;q^2)_{\lfloor\frac{n+1}2\rfloor}(-q^{1+2\lfloor\frac{n+1}2\rfloor}a^2;q^2)_{\lfloor\frac n2\rfloor}}{(qa^2;q^2)_{\lfloor\frac{n+1}2\rfloor}(-q^{1+2\lfloor\frac{n+1}2\rfloor};q^2)_{\lfloor\frac n2\rfloor}}\\
&=a^{2(n-2\lfloor\frac n2\rfloor)}
\frac{(q;q^2)_{\lfloor\frac{n+1}2\rfloor}(-q;q^2)_{\lfloor\frac{n+1}2\rfloor}(-q^{1+2\lfloor\frac{n+1}2\rfloor}a^2;q^2)_{\lfloor\frac n2\rfloor}}{(q^2;q^2)_n(qa^2;q^2)_{\lfloor\frac{n+1}2\rfloor}}\\
&=a^{2(n-2\lfloor\frac n2\rfloor)}
\frac{(q^2;q^4)_{\lfloor\frac{n+1}2\rfloor}(-q^{1+2\lfloor\frac{n+1}2\rfloor}a^2;q^2)_{\lfloor\frac n2\rfloor}}{(q^2;q^2)_n(qa^2;q^2)_{\lfloor\frac{n+1}2\rfloor}}\\
&=a^{2(n-2\lfloor\frac n2\rfloor)}
\frac{(-q^{1+2\lfloor\frac{n+1}2\rfloor}a^2;q^2)_{\lfloor\frac n2\rfloor}}{(q^4;q^4)_{\lfloor\frac{n}2\rfloor}(qa^2;q^2)_{\lfloor\frac{n+1}2\rfloor}}\\
&=a^{2(n-2\lfloor\frac n2\rfloor)}
\frac{(-qa^2;q^2)_n}{(q^4;q^4)_{\lfloor\frac{n}2\rfloor}(qa^2;q^2)_{\lfloor\frac{n+1}2\rfloor}(-qa^2;q^2)_{\lfloor\frac{n+1}2\rfloor}}\\
&=a^{2(n-2\lfloor\frac n2\rfloor)}
\frac{(-qa^2;q^2)_n}{(q^4;q^4)_{\lfloor\frac{n}2\rfloor}(q^2a^4;q^4)_{\lfloor\frac{n+1}2\rfloor}}.
\end{align*}

\noindent
\underline{Case 2}: $(a,b,c)\mapsto (a,-a,a^2)$:

By Theorem~\ref{thm:n4} we have
\begin{align*}
a_n&=\frac{(-1;q^2)_n}{(q^2;q^2)_n}(-1)^n\frac{(q^n+q^{-n})}2
\frac{(q;q^2)_{\lfloor\frac{n+1}2\rfloor}(-q^{2\lfloor\frac{n+1}2\rfloor}a^2;q^2)_{\lfloor\frac n2\rfloor}}{(qa^2;q^2)_{\lfloor\frac{n}2\rfloor}(-q^{2+2\lfloor\frac{n}2\rfloor};q^2)_{\lfloor\frac{n+1}2\rfloor}}\\
&=(-1)^n q^{-n}
\frac{(-q^2;q^2)_n(q;q^2)_{\lfloor\frac{n+1}2\rfloor}(-q^{2\lfloor\frac{n+1}2\rfloor}a^2;q^2)_{\lfloor\frac n2\rfloor}}{(q^2;q^2)_n(qa^2;q^2)_{\lfloor\frac{n}2\rfloor}(-q^{2+2\lfloor\frac{n}2\rfloor};q^2)_{\lfloor\frac{n+1}2\rfloor}}\\
&=(-1)^n q^{-n}
\frac{(-q^2;q^2)_{\lfloor\frac{n}2\rfloor}(q;q^2)_{\lfloor\frac{n+1}2\rfloor}(-a^2;q^2)_{n}}{(q^2;q^2)_n(qa^2;q^2)_{\lfloor\frac{n}2\rfloor}(-a^2;q^2)_{\lfloor\frac {n+1}2\rfloor}}\\
&=(-q)^{-n}
\frac{(q;-q)_n(-a^2;q^2)_{n}}{(q^2;q^2)_n(-a^2;-q)_n},
\end{align*}
which remarkably involves base $-q$ (``minus $q$", which is not a mistake!) instead of the usual base $q$.

\noindent
\underline{Case 3}: $(a,b,c)\mapsto (a,-a,q^{-1} a^2)$:

By Corollary~\ref{cor:310s} we have
\begin{align*}
a_n&=\frac{(-q^{-1};q^2)_n}{(q^2;q^2)_n}
(-1)^n\frac{(1+q^{1-2n})}{(1+q)}
\frac{(q;q^2)_{\lfloor\frac{n+1}2\rfloor}(-q^{-1+2\lfloor\frac{n+1}2\rfloor}a^2;q^2)_{\lfloor\frac n2\rfloor}}{(q^{-1}a^2;q^2)_{\lfloor\frac{n+1}2\rfloor}(-q^{1+2\lfloor\frac{n+1}2\rfloor};q^2)_{\lfloor\frac n2\rfloor}}\\
&=(-1)^n q^{-2n}
\frac{(-q;q^2)_n(q;q^2)_{\lfloor\frac{n+1}2\rfloor}(-q^{-1+2\lfloor\frac{n+1}2\rfloor}a^2;q^2)_{\lfloor\frac n2\rfloor}}{(q^2;q^2)_n(q^{-1}a^2;q^2)_{\lfloor\frac{n+1}2\rfloor}(-q^{1+2\lfloor\frac{n+1}2\rfloor};q^2)_{\lfloor\frac n2\rfloor}}\\
&=(-1)^n q^{-2n}
\frac{(-q;q^2)_{\lfloor\frac{n+1}2\rfloor}(q;q^2)_{\lfloor\frac{n+1}2\rfloor}(-q^{-1}a^2;q^2)_{n}}{(q^2;q^2)_n(q^{-1}a^2;q^2)_{\lfloor\frac{n+1}2\rfloor}(-q^{-1}a^2;q^2)_{\lfloor\frac{n+1}2\rfloor}}\\
&=(-1)^n q^{-2n}
\frac{(q^2;q^4)_{\lfloor\frac{n+1}2\rfloor}(-q^{-1}a^2;q^2)_{n}}{(q^2;q^2)_n(q^{-2}a^4;q^4)_{\lfloor\frac{n+1}2\rfloor}}\\
&=(-1)^n q^{-2n}
\frac{(-q^{-1}a^2;q^2)_{n}}{(q^4;q^4)_{\lfloor\frac{n}2\rfloor}(q^{-2}a^4;q^4)_{\lfloor\frac{n+1}2\rfloor}}
.
\end{align*}
With these formulas for $a_n$ we can write out \eqref{ex317} in the respective cases.
For the first case, after replacing $q$ by $q^{\frac 12}$ amd $a$ by $\sqrt{a}$, this yields \eqref{eq:thm4151}.
In the second case, after replacing $z$ by $-qz$, and subsequently $q$ by $q^{\frac 12}$ and $a$ by $\sqrt{a}$, this yields \eqref{eq:thm4152}.
Finally, In the third case, after replacing $z$ by $-q^2z$, and subsequently $q$ by $q^{\frac 12}$ and $a$ by $\sqrt{a}$, this yields \eqref{eq:thm4153}.
\end{proof}
We offer another application of Lemma~\ref{zeroprop}, similar to those applications in Theorem~\ref{Thm417} but depending on a additional parameter and where the right-hand side involves a sum of two series.

\begin{thm}
Let $0<|q|<1$, $a,b,z\in\CCast$ such that $|z|<1$. Then one has the following product formula:
\begin{eqnarray}
&&\hspace{-0.5cm}\qhyp21{q^2a,q^2b}{q^3ab}{q^2,z}
\qhyp21{a,b}{q^3ab}{q^2,qz}\notag\\
&&\hspace{0.0cm}=\frac{(1-q)(1-qab)}{(1-qa)(1-qb)}
\qhyp43{qa,qb,\ppm q\sqrt{ab}}{qab,\ppm \sqrt{q^3ab}}{q,z}+
\frac{q\,(1-a)(1-b)}{(1-qa)(1-qb)}
\qhyp43{qa,qb,\ppm q\sqrt{ab}}{q^2ab,\ppm \sqrt{q^3ab}}{q,z}
.\label{eq:thm418}
\end{eqnarray}
\end{thm}
\begin{proof}
The analysis is similar to that in the proof of Theorem~\ref{Thm417}.
Again, we apply Lemma~\ref{zeroprop} but now with the specialization $c=q^2ab$.
The coefficients in \eqref{ex317prop} can be computed by the $c=q^2ab$ case of \eqref{eq:ancoeff},
thus
\begin{equation*}
a_n=\qhyp32{a,b,q^{-n}}{q^2ab,q^{-n}}{q,q},
\end{equation*}
which is specific terminating 2-balanced $_3\phi_2$ series.
The latter can be evaluated by the $c=q^2ab$ case of the general
terminating 2-balanced $_3\phi_2$ summation, i.e.,
\begin{equation*}
\qhyp32{a,b,q^{-n}}{c,\frac{q^{2-n}ab}c}{q,q}=
\frac{(1-\frac cq)(1-\frac c{qab})}{(1-\frac c{qa})(1-\frac c{qb})}
\frac{(\frac c{qa},\frac c{qb};q)_{n}}{(\frac cq,\frac c{qab};q)_{n}}
+\frac{c\,(1-a)(1-b)}{qab\,(1-\frac c{qa})(1-\frac c{qb})}
\frac{(\frac c{qa},\frac c{qb};q)_{n}}{(c,\frac c{qab};q)_{n}},
\end{equation*}
which itself readily follows from the contiguous relation \cite[(2.3)]{Krattenthaler1993},
in combination with the $q$-Pfaff--Saalsch\"utz summation \cite[(II.12)]{GaspRah}.
We thus obtain
\begin{equation*}
a_n=
\frac{(1-q)(1-qab)}{(1-qa)(1-qb)}
\frac{(qa,qb;q)_{n}}{(q,qab;q)_{n}}
+\frac{q\,(1-a)(1-b)}{(1-qa)(1-qb)}
\frac{(qa,qb;q)_{n}}{(q,q^2ab;q)_{n}}.
\end{equation*}
Writing out \eqref{ex317} for $c=q^2ab$ and this choice of $a_n$ gives \eqref{eq:thm418}.
\end{proof}
One can also derive other product formulas using the method employed in this section.
One interesting example is obtained by the following slight modification
of the analysis that was used to derive Theorem~\ref{thm39}.

\begin{rem}
If one follows the same procedure as in the proof
of Theorem~\ref{thm39}, but instead takes
$(a,b,c)\mapsto(\frac{a}{q},qb,qab)$ in Lemma \ref{secondprop},
one obtains a different $a_n$, namely
\begin{equation*}
a_n=\frac{(q^2b,\frac{a}{q};q)_n}{(q,ab;q)_n}.
\end{equation*}
This is obtained using a different instance of the $q$-Pfaff--Saalsch\"utz summation \cite[(II.12)]{GaspRah}.
Following the same procedure as in the proof of 
Theorem \ref{thm39} produces
\begin{eqnarray*}
&&\hspace{-4cm}\qhyp21{\frac{a^2}{q},q^3b^2}{qa^2b^2}{q^2,z}\qhyp21{\frac{a^2}{q},qb^2}{qa^2b^2}{q^2,qz}=\qhyp43{\frac{a^2}{q},q^2b^2,\ppm ab}{a^2b^2,\ppm q^\frac12 ab}{q,z},
\end{eqnarray*}
where $|q|<1$ and $|z|<1$ for convergence.
However, this result can be seen to be equivalent
to Nassrallah's second result \eqref{NassrallahB} by suitably shifting $a$ and $b$.
\end{rem}

\section{Conclusion}\label{sec:con}
In this paper we focused on known and new results related to the
Ismail--Wilson generating function for Askey--Wilson
polynomials~\eqref{AWgf}. In Theorem~\ref{thm:32pf} we were able to give
an extension of \eqref{AWgf} that contains an extra parameter $u$
(and reduces to the former if $u=0$),
which we used to deduce a quadruple basic hypergeometric summation formula
in Corollary~\ref{cor:1.3}.
Independently, by specializing the Ismail--Wilson generating
function \eqref{AWgf} in different ways, such that the Askey--Wilson
polynomials that appear in the expansion simplify (by virtue of existing
summations for terminating balanced $_4\phi_3$ series), we were able to
identify a number of formulas for products of $_2\phi_1$ basic
hypergeometric series (some of them known, some of them new) --
for which we also provided corresponding integral representations.
In our effort to put the results obtained into context, we also
described some alternative ways to arrive at several of the results,
such as invoking Cayley--Orr type expansion formulas. By applying
meaningful variations of those derivations, we succeeded in finding 
yet further new identities. Altogether, we believe that the material 
provided in this paper forms a valuable contribution to identities satisfied
by (specialized) Askey--Wilson polynomials.


\section*{Acknowledgements}
We would like to thank the reviewers for their helpful remarks.
In addition we would like to thank Slobodan Damjanovic for pointing us to
relevant literature and how certain results are connected with each other
that we otherwise would have missed.
The second author was partially supported by
Austrian Science Fund FWF
\href{https://doi.org/10.55776/P32305}{10.55776/P32305}.



\def\cprime{$'$} \def\dbar{\leavevmode\hbox to 0pt{\hskip.2ex \accent"16\hss}d}

\end{document}